\documentclass[10pt]{amsart}
\usepackage[margin=1in]{geometry}
\usepackage{amsmath, amsthm, amssymb, color, url}
\usepackage{diagrams}
\usepackage{lscape}
\usepackage{hyperref}
\hypersetup{
    colorlinks,
    citecolor=blue,
    filecolor=black,
    linkcolor=black,
    urlcolor=black
}
\usepackage[disable]{todonotes}
\usepackage[final]{showkeys}

\renewcommand{\oe}{\overline{e}}
\newcommand{\of}{\overline{f}}

\newcommand{\CC}{\mathbb{C}}
\newcommand{\RR}{\mathbb{R}}

\newcommand{\Vect}{\text{Vect}}
\newcommand{\gl}{\mathfrak{gl}}
\newcommand{\oL}{\overline{L}}
\newcommand{\uu}{\mathfrak{u}}

\newcommand{\OO}{\mathcal{O}}
\newcommand{\Tr}{\text{Tr}}
\newcommand{\Sym}{\text{Sym}}
\newcommand{\dr}{\text{dr}}
\newcommand{\eps}{\varepsilon}
\newcommand{\sym}{\text{sym}}
\newcommand{\diag}{\text{diag}}
\newcommand{\wtilde}{\widetilde}

\newcommand{\GT}{\mathrm{GT}}
\newcommand{\trig}{\text{trig}}

\newcommand{\Ad}{\mathrm{Ad}}
\newcommand{\FF}{\mathcal{F}}
\newcommand{\Res}{\text{Res}}
\newcommand{\Vol}{\text{Vol}}
\newcommand{\ad}{\text{ad}}
\newcommand{\End}{\text{End}}
\renewcommand{\tt}{\mathfrak{t}}
\renewcommand{\sl}{\mathfrak{sl}}

\newcommand{\pp}{\mathfrak{p}}
\newcommand{\gt}{\text{gt}}
\newcommand{\bL}{\overline{L}}
\newcommand{\BB}{\mathcal{B}}
\newcommand{\BG}{\mathfrak{B}}
\newcommand{\wU}{\widetilde{U}}

\newcommand{\ZZ}{\mathbb{Z}}
\newcommand{\height}{\text{ht}}

\newcommand{\bb}{\mathfrak{b}}
\newcommand{\ev}{\text{ev}}
\newcommand{\cP}{\mathcal{P}}
\newcommand{\cR}{\mathcal{R}}
\newcommand{\oLt}{\oL^\trig}

\theoremstyle{definition}

\newtheorem{theorem}{Theorem}[section]

\newtheorem{corr}[theorem]{Corollary}
\newtheorem{lemma}[theorem]{Lemma}
\newtheorem{prop}[theorem]{Proposition}
\newtheorem*{remark}{Remark}

\numberwithin{equation}{section}

\begin{document}

\title{A new integral formula for Heckman-Opdam hypergeometric functions}
\author{Yi Sun}
\date{\today}
\email{yisun@math.mit.edu}

\begin{abstract}
We provide Harish-Chandra type formulas for the multivariate Bessel functions and Heckman-Opdam hypergeometric functions as representation-valued integrals over dressing orbits.  Our expression is the quasi-classical limit of the realization of Macdonald polynomials as traces of intertwiners of quantum groups given by Etingof-Kirillov~Jr. in \cite{EK}. Integration over the Liouville tori of the Gelfand-Tsetlin integrable system and adjunction for higher Calogero-Moser Hamiltonians recovers and gives a new proof of the integral realization over Gelfand-Tsetlin polytopes which appeared in the recent work \cite{BG} of Borodin-Gorin on the $\beta$-Jacobi corners ensemble.
\end{abstract}
\maketitle

\tableofcontents

\makeatletter
\providecommand\@dotsep{5}
\makeatother

\section{Introduction}

The Heckman-Opdam hypergeometric functions are a family of real-analytic symmetric functions introduced by Heckman-Opdam in \cite{HO, He, Opd, Opd2} as joint eigenfunctions of the trigonometric Calogero-Moser integrable system.  The latter is a quasi-classical limit of the Macdonald-Ruijsenaars integrable system, and in \cite{BG}, Borodin-Gorin realized the Heckman-Opdam hypergeometric function as a limit of the Macdonald polynomials under the quasi-classical scaling.  By applying their limit transition to Macdonald's branching rule, they obtained a new formula for the Heckman-Opdam hypergeometric functions as an integral over Gelfand-Tsetlin polytopes.

The purpose of the present work is to provide new Harish-Chandra type integral formulas for the Heckman-Opdam hypergeometric functions as representation-valued integrals over dressing orbits of $U_N$.  Our formulas are the quasi-classical limits of the expression given by Etingof-Kirillov~Jr. in \cite{EK} for Macdonald polynomials as representation-valued traces of $U_q(\gl_N)$-intertwiners.  In this limit, traces over irreducible representations become integrals with respect to Liouville measure on the corresponding dressing orbit.

Integrating our formulas over Liouville tori of the Gelfand-Tsetlin integrable system yields an expression for Heckman-Opdam hypergeometric functions as an integral of $U_N$-matrix elements over the Gelfand-Tsetlin polytope.  We identify these matrix elements as an application of higher Calogero-Moser Hamiltonians to an explicit kernel.  Taking adjoints of these Hamiltonians recovers and gives a new proof of the formula of \cite{BG}.  Our techniques involve a relation between spherical parts of rational Cherednik algebras of different rank which is of independent interest.

In the remainder of the introduction, we summarize our motivations, give precise statements of our results, and explain how they relate to other recent work.

\subsection{Heckman-Opdam hypergeometric functions}

Fix a complex number $k$ and a positive integer $N$.  The rational and trigonometric Calogero-Moser integrable systems in the variables $\{\lambda_i\}_{1 \leq i \leq N}$ are the quantum integrable systems with quadratic Hamiltonians
\begin{align*}
L_{p_2}(k) &= \sum_i \partial_i^2 + 2k (1 - k) \sum_{i < j} \frac{1}{(\lambda_i - \lambda_j)^2} \text{ and} \\
L_{p_2}^\trig(k) &= \sum_i \partial_i^2 + k (1 - k) \sum_{i < j} \frac{1}{2\sinh^2\left(\frac{\lambda_i - \lambda_j}{2}\right)}.
\end{align*}
They are completely integrable systems, meaning that $L_{p_2}(k)$ and $L_{p_2}^\text{trig}(k)$ fit into families $L_p(k)$ and $L_p^\text{trig}(k)$ of commuting Hamiltonians defined for each symmetric polynomial $p$. Define conjugated versions of these Hamiltonians by 
\begin{align} \label{eq:conj-cm-ham-def}
\overline{L}_p(k) &= \Delta(\lambda)^{-k} \circ L_p(k) \circ \Delta(\lambda)^{k}\\
\overline{L}_p^\trig(k) &= e^{\frac{(N - 1)k}{2} \sum_i \lambda_i}\Delta(e^\lambda)^{-k} \circ L_p^\trig(k) \circ e^{-\frac{(N - 1)k}{2} \sum_i \lambda_i} \Delta(e^\lambda)^{k}, \label{eq:conj-tcm-ham-def}
\end{align}
where for a set of variables $x$, we denote by $\Delta(x)$ the Vandermonde determinant $\Delta(x) = \prod_{i < j} (x_i - x_j)$.  For each $s = (s_1, \ldots, s_N)$, the hypergeometric system corresponding to $s$ was introduced in \cite{HO, He, Opd, Opd2} as
\begin{equation} \label{eq:hyper}
\bL_p^{\text{trig}}(k) \FF_k(\lambda, s) = p(s) \FF_k(\lambda, s).
\end{equation}
Let $\rho$ be the weight $\rho = (\frac{N - 1}{2}, \ldots, \frac{1 - N}{2})$.  The following characterization was given of certain joint eigenfunctions of this system known as Heckman-Opdam hypergeometric functions.

\begin{theorem}[\cite{HS, Opd3}] \label{thm:ho-uniq}
The hypergeometric system (\ref{eq:hyper}) has a unique symmetric real-analytic solution $\FF_k(\lambda, s)$ normalized so that the leading term of its series expansion in $\lambda$ is
\[
\frac{\Gamma(Nk) \cdots \Gamma(k)}{\Gamma(k)^N} \prod_{i < j} \prod_{a = 0}^{k-1} (s_i - s_j + a)^{-1} e^{(\lambda, s - k \rho)}.
\]
This $\FF_k(\lambda, s)$ extends to a holomorphic function of $\lambda$ on a symmetric tubular neighborhood of $\RR^n \subset \CC^n$. 
\end{theorem}

The corresponding rational degenerations are a family of symmetric real-analytic joint eigenfunctions $\BB_k(\lambda, s)$ of $\bL_p(k)$ satisfying
\begin{equation} \label{eq:bessel-ef}
\bL_p(k) \BB_k(\lambda, s) = p(s) \BB_k(\lambda, s)
\end{equation}
and normalized so that $\BB_k(\lambda, 0) = 1$.  They are known as multivariate Bessel functions and have been studied in \cite{Dun, Jeu, Opd4, OO, GK, FoRa}.

\subsection{Poisson-Lie group structure on $\uu_N$ and $U_N$}

The Lie algebra $\gl_N = \gl_N(\CC)$ has real Iwasawa decomposition $\gl_N = \uu_N \oplus \bb_N$ with $\bb_N \simeq \uu_N^*$.  Let $\tt_N \subset \uu_N$ be the Cartan subalgebra. We identify $\uu_N^*$ with $\pp_N$, the trivial Lie algebra of $N \times N$ Hermitian matrices by the map $x \mapsto \frac{1}{2}(x + x^*)$.  Equip $\pp_N$ with the Kirillov-Kostant-Souriau Poisson structure, and denote the coadjoint orbit of a diagonal matrix $\lambda \in \pp_N$ by $\OO_\lambda$.  We will use $\lambda$ interchangeably for the diagonal matrix and its sequence of diagonal entries.  Denote the symplectic form and Liouville measure on $\OO_\lambda$ by $\omega_\lambda$ and $d\mu_\lambda$, respectively, and let $\CC[\bb_N]$ be the corresponding Poisson algebra.

In the corresponding Iwasawa decomposition $GL_N = U_N B_N$ for the group, give $U_N$ the Lu-Weinstein Poisson-Lie structure (see \cite{LW}) so that $B_N$ is the dual Poisson-Lie group to $U_N$.  Let $T_N \subset U_N$ denote the diagonal torus.  Identify $B_N$ with the Poisson manifold $P_N^+$ of $N \times N$ positive definite Hermitian matrices via $\sym(b) = (b^*b)^{1/2}$ so that $\sym$ intertwines the dressing and conjugation actions of $U_N$ on $B_N$ and $P_N^+$.  For $\Lambda = e^\lambda \in P_N^+$, denote by $\OO_\Lambda$, $\omega_\Lambda$, and $d\mu_\Lambda$ the dressing orbit containing $\Lambda$, its symplectic form, and its Liouville measure.  Let $\CC[B_N]$ and $\CC[\OO_\Lambda]$ denote the corresponding Poisson algebras; these algebras possess a $\star$-structure given by complex conjugation on each matrix element.

\subsection{The main results}

Restrict now to the case of positive integer $k$.  Let $W_{k-1}$ denote the $U_N$-representation
\[
L_{((k - 1)(N - 1), - (k - 1), \ldots, -(k - 1))} = \Sym^{(k-1)N} \CC^N \otimes (\det)^{-(k - 1)},
\]
and choose an isomorphism $W_{k-1}[0] \simeq \CC \cdot w_{k-1}$ for some $w_{k-1} \in W_{k-1} [0]$ which spans the $1$-dimensional zero weight space $W_{k-1}[0]$. Let $f_{k - 1}: \OO_\lambda \to W_{k - 1}$ and $F_{k - 1}: \OO_\Lambda \to W_{k - 1}$ denote the unique $U_N$-equivariant maps such that $f_{k-1}(\lambda) = F_{k - 1}(\Lambda) = w_{k - 1}$.  Our main results are Theorems \ref{thm:rational} and \ref{thm:trig-integral}, which realize the multivariate Bessel functions and Heckman-Opdam hypergeometric functions as representation-valued integrals over coadjoint and dressing orbits under the identification of $W_{k - 1}[0] \simeq \CC \cdot w_{k - 1}$ with $\CC$.

{\renewcommand{\thetheorem}{\ref{thm:rational}}
\begin{theorem}
The multivariate Bessel function $\BB_k(\lambda, s)$ admits the integral representation
\[
\BB_k(\lambda, s) = \frac{\Gamma(Nk) \cdots \Gamma(k)}{ \Gamma(k)^{N}\prod_{i < j} (\lambda_i - \lambda_j)^k \prod_{i < j} (s_i - s_j)^{k - 1}} \int_{X \in \OO_\lambda} f_{k - 1}(X) e^{\sum_{l = 1}^N s_l X_{ll}} d\mu_\lambda.
\]
\end{theorem}
\addtocounter{theorem}{-1}}

{\renewcommand{\thetheorem}{\ref{thm:trig-integral}}
\begin{theorem}
The Heckman-Opdam hypergeometric function $\FF_k(\lambda, s)$ admits the integral representation
\[
\FF_k(\lambda, s) = \frac{\Gamma(Nk) \cdots \Gamma(k)}{\Gamma(k)^{N} \prod_{i < j} (e^{\frac{\lambda_i - \lambda_j}{2}} - e^{- \frac{\lambda_i - \lambda_j}{2}})^k \prod_{a = 1}^{k - 1}\prod_{i < j}(s_i - s_j - a)} \int_{X \in \OO_\Lambda} F_{k - 1}(X) \prod_{l = 1}^N \left(\frac{\det(X_l)}{\det(X_{l - 1})}\right)^{s_l} d\mu_\Lambda,
\]
where $X_l$ is the principal $l \times l$ submatrix of $X$.
\end{theorem}
\addtocounter{theorem}{-1}}

\begin{remark}
The $k = 1$ case of the integral of Theorem \ref{thm:rational} is the HCIZ integral of \cite{H1, H2, IZ}.  It also generalizes the construction of \cite{GK}, where a similar construction is made for $k = 1, 2$.
\end{remark}

\subsection{Existing integral formulas and connection to $\beta$-Jacobi corners ensemble}

Scalings of Heckman-Opdam functions appeared in the work \cite{BG} of Borodin-Gorin on the $\beta$-Jacobi corners ensemble, where they were obtained as a certain scaling limit of the Macdonald polynomials $P_\mu(x; q, t)$.  For $\lambda_1 \geq \cdots \geq \lambda_N \in \RR^N$, define the Gelfand-Tsetlin polytope to be 
\[
\GT_\lambda := \{(\mu^l_i)_{1 \leq i \leq l, 1 \leq l < N} \mid \mu^{l+1}_i \geq \mu^l_i \geq \mu^{l+1}_{i+1}\},
\]
where we take $\mu^N_i = \lambda_i$.  A point $\{\mu^l_i\}$ in $\GT_\lambda$ is called a Gelfand-Tsetlin pattern.  To state the result of \cite{BG}, we define the integral formulas
\begin{equation} \label{eq:int-rat}
\phi_k(\lambda, s) = \Gamma(k)^{-\frac{N(N - 1)}{2}}\int_{\mu \in \GT_\lambda}\!\!\!\!\!\!\!\!\!\! e^{\sum_{l = 1}^N s_l(\sum_i \mu^l_i - \sum_i \mu^{l-1}_i)}\! \prod_{l = 1}^{N-1} \frac{\prod_{i = 1}^l \prod_{j = 1}^{l+1} |\mu^l_i - \mu^{l+1}_j|^{k-1}}{\prod_{i < j} |\mu^l_i - \mu^l_j|^{k - 1} \prod_{i < j} |\mu^{l+1}_i - \mu^{l+1}_j|^{k-1}} \prod_{i = 1} d\mu^l_i
\end{equation}
and
\begin{align}\label{eq:ho-scale}
\Phi_k(\lambda, s) &= \Gamma(k)^{-\frac{N(N - 1)}{2}} \int_{\mu \in \GT_\lambda} e^{\left(\sum_{l = 1}^N s_l \left(\sum_{i = 1}^l \mu^l_i - \sum_{i = 1}^{l-1} \mu^{l-1}_i\right)\right)}\\
&\phantom{===} \prod_{l = 1}^{N-1} \frac{\prod_{i = 1}^l \prod_{j = 1}^{l+1} |e^{\mu^l_i} - e^{\mu^{l+1}_j}|^{k - 1}}{\prod_{i < j} |e^{\mu^l_i} - e^{\mu^l_j}|^{k - 1} \prod_{i < j} |e^{\mu^{l+1}_i} - e^{\mu^{l+1}_j}|^{k - 1}} \prod_{l = 1}^{N-1} e^{-(k - 1)\sum_{i = 1}^l \mu^l_i} \prod_i d\mu^l_i, \nonumber
\end{align}
where (\ref{eq:int-rat}) is a rational degeneration of (\ref{eq:ho-scale}).  In \cite{GK}, the formula (\ref{eq:int-rat}) was related to the multivariate Bessel functions as follows; a related approach was given for $k = 1/2, 1, 2$ in \cite[Appendix C]{FoRa}.

\begin{theorem}[{\cite[Section V]{GK}}] \label{thm:bess-bg-int}
For positive real $k > 0$ and $\lambda_1 > \cdots > \lambda_N$, the multivariate Bessel function is given by 
\[
\BB_k(\lambda, s) = \frac{\Gamma(Nk) \cdots \Gamma(k)}{\Gamma(k)^{N}} \frac{\phi_k(\lambda, s)}{\prod_{i < j}(\lambda_i - \lambda_j)^k}.
\]
\end{theorem}

\begin{remark}
We have adjusted the normalization of $\BB_k(\lambda, s)$ in Theorem \ref{thm:bess-bg-int} from \cite{GK} so that $\BB_k(\lambda, 0) = 1$.
\end{remark}

In the trigonometric setting, the integral formula of (\ref{eq:ho-scale}) was realized by Borodin-Gorin as a scaling limit of Macdonald polynomials.  Applying this scaling to the eigenfunction relation for Macdonald polynomials, they showed that $\Phi_k(\lambda, s)$ was an eigenfunction of the quadratic Calogero-Moser Hamiltonian $L_{p_2}^\text{trig}(k - 1)$.  Together with some arguments which we detail in Subsection \ref{sec:ho-fn} for $k$ a positive integer, this relates $\Phi_k(\lambda, s)$ to $\FF_k(\lambda, s)$.

\begin{theorem}[{\cite[Proposition 6.2]{BG}}] \label{thm:bg}
For any positive real $k > 0$, $\Phi_k(\lambda, s)$ is the following scaling limit of Macdonald polynomials
\[
\Phi_k(\lambda, s) = \lim_{\eps \to 0} \eps^{k N (N - 1)/2} P_{\lfloor \eps^{-1} \lambda \rfloor}(e^{\eps s}; e^{-\eps}, e^{-k\eps}).
\]
\end{theorem}

\begin{theorem}[{\cite[Definition 6.1 and Proposition 6.3]{BG}}] \label{thm:bg-int}
For any positive real $k > 0$ and $\lambda_1 > \cdots > \lambda_N$, the Heckman-Opdam hypergeometric function is given by
\[
\FF_k(\lambda, s) = \frac{\Gamma(Nk) \cdots \Gamma(k)}{\Gamma(k)^{N}} \frac{\Phi_k(\lambda, s)}{\prod_{i < j} (e^{\frac{\lambda_i - \lambda_j}{2}} - e^{-\frac{\lambda_i - \lambda_j}{2}})^k}.
\]
\end{theorem}

\begin{remark}
The integral formulas of Theorems \ref{thm:bess-bg-int} and \ref{thm:bg-int} are stated only for $\lambda_1 > \cdots > \lambda_N$.  We may extend them to $\{\lambda_i \neq \lambda_j\}$ by imposing that $\FF_k(\lambda, s)$ and $\BB_k(\lambda, s)$ are symmetric in $\lambda$.  Under this extension, by taking limits of relevant normalizations of (\ref{eq:int-rat}) and (\ref{eq:ho-scale}) we may show that the expressions of Theorems \ref{thm:bess-bg-int} and \ref{thm:bg-int} extend to $\lambda \in \RR^N$.  We give such arguments for the trigonometric case when $k > 0$ is a positive integer in Subsection \ref{sec:ho-fn}.
\end{remark}

\begin{remark}
The main result of \cite[Theorem 6.3]{Kaz} gives for each Weyl chamber a contour integral formula for a solution to the hypergeometric system (\ref{eq:hyper}) holomorphic in that Weyl chamber.  These formulas have the same integrand as the integral of Theorem \ref{thm:bg-int} but contours which are different for each Weyl chamber. 
\end{remark}

\subsection{Realization via quasi-classical limit of quantum group intertwiners}

The formula of Theorem \ref{thm:trig-integral} is the quasi-classical limit of the trace of an intertwiner of quantum group representations.  We will give a second approach to its proof using this theory; when combined with our first proof of Theorem \ref{thm:trig-integral}, this provides a new proof of Theorem \ref{thm:bg-int} from \cite{BG}.  Our approach proceeds via the degeneration of $U_q(\gl_N)$-representations; we summarize the main idea in this subsection and give full details in Section \ref{sec:qcl}.  

For a dominant integral weight $\lambda$, let $L_\lambda$ denote the corresponding highest weight irreducible representation of $U_q(\gl_N)$.  Let $\rho = \frac{1}{2} \sum_{\alpha > 0} \alpha$ be half the sum of the positive roots.  In \cite{EK}, it was shown that there exists a unique intertwiner $\Phi_\lambda^N : L_{\lambda + (k - 1)\rho} \to L_{\lambda + (k - 1) \rho} \otimes W_{k - 1}$ of $U_q(\gl_N)$-representations such that the highest weight vector $v_{\lambda + (k - 1)\rho} \in L_{\lambda + (k - 1)\rho}$ is mapped to 
\[
\Phi_\lambda^N(v_{\lambda + (k-1)\rho}) = v_{\lambda + (k - 1)\rho} \otimes w_{k - 1} + (\text{lower order terms}),
\]
where the lower order terms have weight less than $\lambda + (k - 1)\rho$ in the $L_{\lambda + (k - 1)\rho}$ tensor factor.  They expressed Macdonald polynomials in terms of these intertwiners in the following theorem.

\begin{theorem}[{\cite[Theorem 1]{EK}}] \label{thm:ek}
The Macdonald polynomial $P_\lambda(x; q^2, q^{2k})$ is given by 
\begin{equation} \label{eq:ek-mac}
P_\lambda(x; q^2, q^{2k}) = \frac{\Tr(\Phi_\lambda^N x^h)}{\Tr(\Phi_0^N x^h)}.
\end{equation}
\end{theorem}

We characterize both sides of (\ref{eq:ek-mac}) under the quasi-classical limit transition of \cite{BG} in the following two results.  Corollary \ref{corr:mac-limit} converts traces of quantum group representations to integrals over dressing orbits to yield an integral expression for the limit.  Theorem \ref{thm:trig-diag} uses the fact that the Macdonald difference operators diagonalize both sides of (\ref{eq:ek-mac}) to show that this limiting integral is diagonalized by the quadratic trigonometric Calogero-Moser Hamiltonian.

{\renewcommand{\thetheorem}{\ref{corr:mac-limit}}
\begin{corr} 
For sequences of dominant integral signatures $\{\lambda_m\}$ and real quantization parameters $\{q_m\}$ so that $\lim_{m \to \infty} q_m \to 1$ and $\lim_{m \to \infty} -2 \log(q_m) \lambda_m = \lambda$ is dominant regular, we have
\[
\lim_{m \to \infty} (-2\log(q_m))^{kN(N - 1)/2} P_{\lambda_m}(q_m^{-2s}; q_m^{2}, q_m^{2k}) =  \frac{\int_{\OO_\Lambda} F_{k - 1}(X) \prod_{l = 1}^N \left(\frac{\det(X_l)}{\det(X_{l-1})}\right)^{s_l} d\mu_\Lambda}{\prod_{a = 1}^{k-1} \prod_{i < j} (s_i - s_j - a)}.
\]
\end{corr}
\addtocounter{theorem}{-1}}

{\renewcommand{\thetheorem}{\ref{thm:trig-diag}}
\begin{theorem}
The trigonometric Calogero-Moser Hamiltonian $\overline{L}_{p_2}^\text{trig}(k)$ is diagonalized on
\[
\frac{1}{\prod_{i < j}(e^{\frac{\lambda_i - \lambda_j}{2}} - e^{-\frac{\lambda_i - \lambda_j}{2}})^k \prod_{a = 1}^{k-1} \prod_{i < j} (s_i - s_j - a)} \int_{\OO_\Lambda} F_{k - 1}(X) \prod_{l = 1}^N \left(\frac{\det(X_l)}{\det(X_{l-1})}\right)^{s_l} d\mu_\Lambda
\]
with eigenvalue $\sum_i s_i^2$.
\end{theorem}
\addtocounter{theorem}{-1}}

\begin{remark}
Combining these two results and our first proof of Theorem \ref{thm:trig-integral} yields a new proof of Theorem \ref{thm:bg-int} which is independent of the results of \cite{BG}. 
\end{remark}

\begin{remark}
In the recent paper \cite{Sun2}, we give a representation theoretic proof of Macdonald's branching rule using a quantum analogue of the results of the present work.  In particular, we identify diagonal matrix elements of $\Phi_\lambda^N$ in the Gelfand-Tsetlin basis with the application of higher Macdonald-Ruijsenaars Hamiltonians to a kernel.  We then apply adjunction to the Etingof-Kirillov~Jr. trace formula to recover the branching rule.  The link established in this paper between the expressions given in Theorem \ref{thm:trig-integral} and \cite{BG} for the Heckman-Opdam hypergeometric functions is the quasiclassical limit of this argument and inspired the approach of \cite{Sun2}.
\end{remark}

\subsection{Outline of method and organization}

We outline our approach.  We first show that the quasi-classical limit of the Etingof-Kirillov~Jr. construction of Macdonald polynomials as traces of $U_q(\gl_N)$-intertwiners corresponds to integrals over dressing orbits of $B_N$ in Corollary \ref{corr:mac-limit} and that these integrals diagonalize the quadratic Calogero-Moser Hamiltonian in Theorem \ref{thm:trig-diag}.  The Gelfand-Tsetlin action on these dressing orbits then defines a classical integrable system whose moment map is the logarithmic Gelfand-Tsetlin map $\GT$ of \cite{FR, AM}.  Integration over the Liouville tori reduces the integral of Theorem \ref{thm:trig-integral} to an integral with respect to the Duistermaat-Heckman measure $\GT_*(d\mu_\Lambda)$ on $\GT_\lambda$, which is the Lebesgue measure.  This yields an integral expression for $\Phi_k(\lambda, s)$ over $\GT_\lambda$.  The new integrand differs from that of Theorem \ref{thm:bg-int}, but we show equality of the integrals by applying adjunction for higher Calogero-Moser Hamiltonians. 

The remainder of this paper is organized as follows.  In Section 2, we give the geometric setup for our integral formulas.  In Section 3, we prove Corollary \ref{corr:mac-limit} and Theorem \ref{thm:trig-diag} by taking the quasi-classical limit of the quantum group setting.  In Section 4, we prove Theorem \ref{thm:rational} in the rational setting, establishing in particular the key Proposition \ref{prop:matrix-element}.  In Section 5, we use Proposition \ref{prop:matrix-element} to give another proof of Theorem \ref{thm:trig-integral} in the trigonometric setting via the formula of \cite{BG}.  In Section 6, we provide proofs for some technical lemmas whose proofs were deferred.

\subsection{Acknowledgments}

The author thanks his Ph.D. advisor P. Etingof for suggesting the problem and providing the idea for the proof of Proposition \ref{prop:matrix-element}.  The author thanks also A. Borodin for helpful discussions and P. Forrester for bringing the references \cite{FoRa, GK} to his attention.  Y.~S. was supported by a NSF graduate research fellowship (NSF Grant \#1122374).

\section{Geometric setup}

\subsection{Notations}

For sets of variables $\{x_i\}$ and $\{y_i\}$, we denote the Vandermonde determinant by $\Delta(x) = \prod_{i < j} (x_i - x_j)$, and the product of differences by $\Delta(x, y) = \prod_{i, j} (x_i - y_j)$.  Denote also the trigonometric Vandermonde by $\Delta^\trig(x) = \prod_{i < j} (e^{\frac{x_i - x_j}{2}} - e^{\frac{x_j - x_i}{2}})$.

\subsection{Gelfand-Tsetlin coordinates}

Define the \textit{Gelfand-Tsetlin map} $\gt: \OO_\lambda \to \GT_\lambda$ by
\[
\gt(X) = \{\lambda_i(X_l)\}_{1 \leq i \leq l, 1 \leq l < N},
\]
where $X_l$ is the principal $l \times l$ submatrix of $X$, and $\lambda_1(X_l) \geq \ldots \geq \lambda_l(X_l)$ are its eigenvalues.   Define the \textit{logarithmic Gelfand-Tsetlin map} $\GT: \OO_\Lambda \to \GT_\lambda$ by 
\[
\GT(X) = \{\log(\lambda_i(X_l))\}_{1 \leq i \leq l, 1 \leq l < N}.
\]
By a theorem of Ginzburg and Weinstein (see \cite{GW}), the Poisson structures we have described on $\bb_N$ and $B_N$ make them isomorphic as Poisson manifolds.  By \cite{AM}, there exists a Ginzburg-Weinstein isomorphism $\bb_N \to B_N$ which intertwines the logarithmic and ordinary Gelfand-Tsetlin maps. In particular, this map restricts to a symplectomorphism $\OO_\lambda \to \OO_\Lambda$.

\subsection{Gelfand-Tsetlin integrable system} \label{ss:gt-int-sys}

Let $T := T_1 \times \cdots \times T_{N-1}$ be a torus of dimension $\frac{N(N - 1)}{2}$, where $\dim T_l = l$.  For $t_l \in T_l$ and $X$ in $\OO_\lambda$ or $\OO_\Lambda$ whose principal $l \times l$ submatrix $X_l$ is diagonalized by $X_l = U_l \Lambda_l U_l^*$, the \textit{Gelfand-Tsetlin} action of $t_l$ on $X$ is defined as
\[
t_l \cdot X = \Ad_{\overline{U_l t_l U_l^*}}(X),
\]
where for $Y_l \in U(l)$, the matrix $\overline{Y_l} \in U_N$ is defined to be the square block matrix
\[
\overline{Y_l} = \left(\begin{array}{ccc|c}
 & & &   0  \\ 
 & Y_l &  &\vdots  \\
& & & 0 \\ \hline
0 & \cdots & 0 &  cI_{N - l}   \end{array}\right),
\]
where $c$ is chosen so that $\overline{Y_l} \in U_N$.  The actions of $T_l$ preserve $l \times l$ principal submatrices and pairwise commute, giving actions of $T$ on $\OO_\lambda$ and $\OO_\Lambda$.  These actions are Hamiltonian with moment maps $\gt$ and $\GT$, respectively, and the corresponding classical integrable system is known as the Gelfand-Tsetlin integrable system (see \cite{AM, GS, FR} for more about this integrable system).

We may use the Gelfand-Tsetlin action to write any $X_0$ in $\gt^{-1}(\mu)$ or $\GT^{-1}(\mu)$ in a special form.   Write $X_0$ as either $u_N \lambda u_N^*$ or $u_N\Lambda u_N^*$ for some unitary matrix $u_N$ and decompose $u_N$ as
\[
u_N = \overline{u}_1 (\overline{u}_1^* \overline{u}_2) \cdots (\overline{u}_{N-1}^* u_N)
\]
for $u_m \in U(m)$ and $v_m := \overline{u}_{m-1}^* u_m$ satisfying either
\[
(v_m \mu^m v_m^*)_{m-1} = \mu^{m-1} \qquad \text{ or } \qquad (v_m e^{\mu^m} v_m^*)_{m-1} = e^{\mu^{m-1}},
\]
where $(M)_{m-1}$ denotes the principal $(m - 1) \times (m - 1)$ submatrix of a matrix $M$.  Lemma \ref{lem:gt-conj} gives a compatibility property between this decomposition and the Gelfand-Tsetlin action.

\begin{lemma} \label{lem:gt-conj}
For any $l \leq m$ and $t_m \in T_m$, we have
\begin{align*}
t_m \cdot \ad_{\overline{v}_l \cdots \overline{v}_N}(\lambda) &= \ad_{\overline{v}_l \cdots \overline{v}_m}(t_m \cdot \ad_{\overline{v}_{m+1} \cdots \overline{v}_N}(\lambda)), \text{ and}\\
t_m \cdot \ad_{\overline{v}_l \cdots \overline{v}_N}(\Lambda) &= \ad_{\overline{v}_l \cdots \overline{v}_m}(t_m \cdot \ad_{\overline{v}_{m+1} \cdots \overline{v}_N}(\Lambda)).
\end{align*}
\end{lemma}
\begin{proof}
By construction, the principal $m \times m$ submatrix of $\ad_{\overline{v}_{m+1} \cdots \overline{v}_N}(\lambda)$ is diagonal, implying that 
\[
t_m \cdot \ad_{\overline{v}_l \cdots \overline{v}_N}(\lambda) = \ad_{\ad_{\overline{v}_l \cdots \overline{v}_m}(t_m)}(\ad_{\overline{v}_{l} \cdots \overline{v}_N}(\lambda)) =  \ad_{\overline{v}_l \cdots \overline{v}_m}(t_m \cdot \ad_{\overline{v}_{m+1} \cdots \overline{v}_N}(\lambda)).
\]
An analogous proof yields the lemma for $\Lambda$ in place of $\lambda$.
\end{proof}

\subsection{Duistermaat-Heckman measures}

The pushforwards $\gt_*(d\mu_\lambda)$ and $\GT_*(d\mu_\Lambda)$ of the Liouville measures on $\OO_\lambda$ and $\OO_\Lambda$ to $\GT_\lambda$ are called Duistermaat-Heckman measures.  Because the Ginzburg-Weinstein isomorphism intertwines the two Gelfand-Tsetlin maps, the two Duistermaat-Heckman measures on $\GT_\lambda$ coincide.  It is known (see \cite[Section 5.6]{GN, Ba, AB}) that the Duistermaat-Heckman measure for the coadjoint orbit $\OO_\lambda$ is proportional to the Lebesgue measure on the Gelfand-Tsetlin polytope.  To compute the normalization constant, we recall Harish-Chandra's formula (see \cite[Theorem 3, Section 3]{Kir})
\begin{equation} \label{eq:hc-form}
\int_{\OO_\lambda} e^{(b, x)} d\mu_\lambda = \frac{\sum_{w \in W} (-1)^w e^{(w \lambda, x)}}{\prod_{i < j}(x_i - x_j)},
\end{equation}
which upon taking $x \to 0$ (via $x = \eps \cdot \rho$ and $\eps \to 0$) shows that 
\[
\Vol(\OO_\lambda) = \frac{\prod_{i < j} (\lambda_i - \lambda_j)}{(N - 1)! \cdots 1!}.
\]
On the other hand, it is known (see \cite[Corollary 3.2]{Ols}) that $\Vol(\GT_\lambda) = \frac{\prod_{i < j} (\lambda_i - \lambda_j)}{(N - 1)! \cdots 1!}$, meaning that $\gt_*(d\mu_\lambda) = 1_{\GT_\lambda} \cdot dx$.  This discussion establishes the following Proposition \ref{prop:dh-compute}.

\begin{prop} \label{prop:dh-compute}
The Duistermaat-Heckman measures $\gt_*(d\mu_\lambda) = \GT_*(d\mu_\Lambda)$ are equal to the Lebesgue measure $dx$ on the Gelfand-Tsetlin polytope.  Explicitly, we have
\[
\gt_*(d\mu_\lambda) = \GT_*(d\mu_\Lambda) = 1_{\GT_\lambda} dx.
\]
\end{prop}

\section{Quasi-classical limits of quantum group intertwiners} \label{sec:qcl}

\subsection{Finite-type quantum group}

Let $U_q(\gl_N)$ be the associative algebra over $\CC(q^{\pm 1/2})$ with generators $e_i, f_i$ for $i = 1, \ldots, N-1$ and $q^{\pm \frac{h_i}{2}}$ for $i = 1, \ldots, N$ and relations
\begin{align*}
q^{\frac{h_i}{2}} e_i q^{-\frac{h_i}{2}} &= q^{\frac{1}{2}} e_i, \qquad q^{\frac{h_i}{2}} e_{i-1} q^{-\frac{h_i}{2}} = q^{-\frac{1}{2}} e_{i-1}, \qquad q^{\frac{h_i}{2}} f_{i} q^{-\frac{h_i}{2}} = q^{-\frac{1}{2}} f_i, \qquad q^{\frac{h_i}{2}} f_{i-1} q^{-\frac{h_i}{2}} = q^{\frac{1}{2}} f_{i-1}\\
[q^{\frac{h_i}{2}}, e_j] = [&q^{\frac{h_i}{2}}, f_j] = 0 \text{ for $j \neq i, i - 1$}, \qquad [e_i, f_j] = \delta_{ij} \frac{q^{h_i - h_{i+1}} - q^{h_{i+1} - h_i}}{q - q^{-1}}, \qquad [e_i, e_j] = [f_i, f_j] = 0 \text{ for $|i - j| > 1$}\\
q^{\frac{h_i}{2}} \cdot q^{-\frac{h_i}{2}} &= 1, \qquad e_i^2 e_j - (q + q^{-1}) e_i e_j e_i + e_j e_i^2 = 0, \qquad f_i^2 f_j - (q + q^{-1}) f_i f_j f_i + f_j f_i^2 = 0 \text{ for $|i - j| = 1$}.
\end{align*}
We take the coproduct on $U_q(\gl_N)$ defined by 
\begin{align*}
\Delta(e_i) &= e_i \otimes q^{\frac{h_{i+1} - h_i}{2}} + q^{\frac{h_i - h_{i+1}}{2}} \otimes e_i\\
\Delta(f_i) &= f_i \otimes q^{\frac{h_{i+1} - h_i}{2}} + q^{\frac{h_i - h_{i+1}}{2}} \otimes f_i \\
\Delta(q^{\frac{h_i}{2}}) &= q^{\frac{h_i}{2}} \otimes q^{\frac{h_i}{2}}
\end{align*}
and the antipode given by 
\[
S(e_i) = -e_i q^{-1}, \qquad S(f_i) = -f_i q, \qquad S(q^{h_i}) = q^{-h_i}.
\]
Taking the $\star$-structure on $U_q(\gl_N)$ given by 
\[
e_i^\star = f_i \qquad \text{ and } \qquad f_i^\star = e_i \qquad \text{ and } \qquad (q^{h_i/2})^\star = q^{h_i/2}
\]
yields the $\star$-Hopf algebra $U_q(\uu_N)$.  Its restriction to the algebra span of $q^{h_i/2}$ is the $\star$-Hopf algebra $U_q(\tt_N)$.

\subsection{Macdonald polynomials and Etingof-Kirillov Jr. construction}

Let $\rho = \left(\frac{N-1}{2}, \ldots, \frac{1 - N}{2}\right)$ and let $e_r$ denote the elementary symmetric polynomial. For a partition $\lambda$, the Macdonald polynomial $P_\lambda(x; q^2, t^2)$ is the joint polynomial eigenfunction with leading term $x^\lambda$ and eigenvalue $e_r(q^{2\lambda} t^{2\rho})$ of the operators
\[
D_{N, x}^r(q^2, t^2) = t^{r(r-N)} \sum_{|I| = r} \prod_{i \in I, j \notin I} \frac{t^2 x_i - x_j}{x_i - x_j} T_{q^2, I},
\]
where $T_{q^2, I} = \prod_{i \in I} T_{q^2, i}$ and $T_{q^2, i}f(x_1, \ldots, x_n) = f(x_1, \ldots, q^2 x_i, \ldots, x_N)$ so that we have
\[
D_{N, x}^r(q^2, t^2) P_\lambda(x; q^2, t^2) = e_r(q^{2\lambda}t^{2\rho}) P_\lambda(x; q^2, t^2).
\]
Note that our normalization of $D_{N, x}^r(q^2, t^2)$ differs from that of \cite{Mac}. In \cite{EK}, Etingof and Kirillov Jr. gave an interpretation of Macdonald polynomials in terms of representation-valued traces of $U_q(\gl_N)$.  For a signature $\lambda$, there exists a unique intertwiner 
\[
\Phi_\lambda^N: L_{\lambda + (k - 1)\rho} \to L_{\lambda + (k - 1)\rho} \otimes W_{k-1}
\]
normalized to send the highest weight vector $v_{\lambda + (k-1)\rho}$ in $L_{\lambda + (k-1)\rho}$ to 
\[
v_{\lambda + (k - 1)\rho} \otimes w_{k-1} + (\text{lower order terms}),
\]
where $(\text{lower order terms})$ denotes terms of weight lower than $\lambda + (k - 1)\rho$ in the first tensor coordinate.  As shown in \cite[Theorem 1]{EK} (reproduced as Theorem \ref{thm:ek}), traces of these intertwiners lie in $W_{k-1}[0] = \CC \cdot w_{k-1}$ and yield Macdonald polynomials when interpreted as scalar functions via the identification $w_{k-1} \mapsto 1$.  The denominator also admits the following explicit form.

\begin{prop}[{\cite[Main Lemma]{EK}}] \label{prop:ek-denom}
On $L_{(k - 1)\rho}$, the trace may be expressed explicitly as
\begin{align*}
\Tr(\Phi_0^Nx^h) &= (x_1 \cdots x_N)^{-\frac{(k-1)(N - 1)}{2}} \prod_{a = 1}^{k - 1} \prod_{i < j} (x_i - q^{2a} x_j).
\end{align*}
\end{prop}

\begin{remark}
Our notation for Macdonald polynomials is related to that of \cite{EK} via $P^{EK}_{\lambda}(x; q, t) = P_{\lambda}(x; q^2, t^2)$.
\end{remark}

\subsection{Braid group action, PBW theorem, and integral forms}

In this section, we define an integral form $U_q'(\gl_N) \subset U_q(\gl_N)$ which will allow us to realize it as a quantum deformation of the Poisson algebra $\CC[B_N]$ in the sense of \cite[Section 11]{CP}. For this, we require Lusztig's braid group action on $U_q(\gl_N)$.  Following \cite{Lus}, the braid group $\BG_N = \left\langle T_1, \ldots, T_{N-1} \mid T_i T_{i+1} T_i = T_{i+1} T_i T_{i+1}\right\rangle$ of type $A_{N - 1}$ acts via algebra automorphisms on $U_q(\gl_N)$ by 
\begin{align*}
T_i(e_i) &= - f_i q^{h_i - h_{i+1}} \qquad T_i(e_{i\pm 1}) = q^{-1} e_{i\pm 1} e_i - e_i e_{i\pm 1} \qquad T_i(e_j) = e_j \text{ for $|i - j| > 1$}\\
T_i(f_i) &= - q^{-h_i + h_{i+1}} e_i \qquad T_i(f_{i\pm 1}) = q f_{i\pm 1} f_i - f_i f_{i\pm 1} \qquad T_i(f_j) = f_j \text{ for $|i - j| > 1$}\\
T_i(q^{h_i/2}) &= q^{h_{i+1}/2} \qquad T_i(q^{h_{i+1}/2}) = q^{h_i/2} \qquad T_i(q^{h_j/2}) = q^{h_j/2} \text{ for $j \neq i, i + 1$}.
\end{align*}
Let $U_q'(\gl_N)$ be the smallest $\CC[q^{\pm 1/2}]$-subalgebra of $U_q(\gl_N)$ containing
\[
\bar{e}_i = (q - q^{-1}) e_i, \qquad \bar{f}_i = (q - q^{-1})f_i, \qquad q^{h_i/2}
\]
and stable under the action of $\BG_N$ described above.  For a choice of simple roots $\{\alpha_1, \ldots, \alpha_{N - 1}\}$ and a fixed decomposition $w_0 = s_{i_1} \cdots s_{i_M}$ of the longest word $w_0$ in $S_N$, let $\beta_l = s_{i_1} \cdots s_{i_{l-1}}(\alpha_l)$ and define 
\[
\overline{e}_{\beta_l} = (q - q^{-1})T_{i_1} \cdots T_{i_{l-1}}(e_l) \text{ and } \overline{f}_{\beta_l} = (q - q^{-1})T_{i_1} \cdots T_{i_{l-1}}(f_l).
\]
By the PBW theorem, $U_q'(\gl_N)$ has a $\CC[q^{\pm 1/2}]$-basis given by monomials
\[
\overline{e}_{\beta_1}^{k_1} \cdots \overline{e}_{\beta_M}^{k_M} q^{h} \overline{f}_{\beta_M}^{l_M} \cdots \overline{f}_{\beta_1}^{l_1}.
\]
Following \cite[Section 10]{CP}, assign such a monomial a degree of 
\[
\deg\Big(\overline{e}_{\beta_1}^{k_1} \cdots \overline{e}_{\beta_M}^{k_M} q^{h} \overline{f}_{\beta_M}^{l_M} \cdots \overline{f}_{\beta_1}^{l_1}\Big) = \Big(k_M, \ldots, k_1, l_1, \ldots, l_M, \sum_{i = 1}^M (k_i + l_i) \height(\beta_i)\Big) \in \ZZ^{2M + 1}_{\geq 0},
\]
where if $\beta = \sum_i c_i \alpha_i$ as the sum of simple roots, its height is $\height(\beta) = \sum_i c_i$.  The algebra $U_q'(\gl_N)$ is a $\ZZ^{2M+1}_{\geq 0}$-filtered algebra under the degree filtration, known as the de Concini-Kac filtration.

\begin{prop}[{\cite[Section 10]{CP}}] \label{prop:cp-gr}
The associated graded of $U_q'(\gl_N)$ under the de Concini-Kac filtration is generated by $\oe_{\beta_i}, \of_{\beta_i}, q^{h_i/2}$ subject to the relations
\begin{align*}
[q^{h_i/2}, q^{h_j/2}] &= 0, \qquad q^{h_i/2}\overline{e}_{\beta_j} = q^{\beta_{j, i}} \overline{e}_{\beta_j} q^{h_i/2}, \qquad q^{h_i/2}\overline{f}_{\beta_j} = q^{-\beta_{j, i}} \overline{f}_{\beta_j} q^{h_i/2}\\
[\overline{e}_{\beta_i}, \overline{f}_{\beta_j}] &= 0, \qquad \overline{e}_{\beta_i}\overline{e}_{\beta_j} = q^{(\beta_i, \beta_j)}\overline{e}_{\beta_j} \overline{e}_{\beta_i} \text{ for $i > j$},  \qquad \overline{f}_{\beta_i}\overline{f}_{\beta_j} = q^{(\beta_i, \beta_j)}\overline{f}_{\beta_j} \overline{f}_{\beta_i} \text{ for $i > j$}.
\end{align*}
\end{prop}

\subsection{Infinitesimal dressing action and Poisson bracket}

In what follows, we will consider functions on $B_N$ pulled back from matrix elements of $P_N^+$ via the map $\sym: B_N \to P_N^+$ as in the statement of Theorem \ref{thm:trig-integral}. The derivative of the dressing action of $U_N$ on $B_N$ yields a map of vector fields $\dr: \uu_N \to \Vect(B_N)$ called the infinitesimal dressing action.  Let $\delta: \CC[B_N] \to \CC[B_N] \otimes \CC[B_N]$ and $S: \CC[B_N] \to \CC[B_N]$ denote the coproduct and antipode on $\CC[B_N]$. In \cite{Lu}, it is shown that the infinitesimal $\uu_N$-action may be realized via the Poisson bracket. 

\begin{prop}[{\cite[Theorem 3.10]{Lu}}] \label{prop:lu-mom}
For $f \in \CC[B_N]$ with $\delta(f) = \sum_i f^{(1)}_i \otimes f^{(2)}_i$, the infinitesimal dressing action of $df|_e \in T_e^*(B_N) \simeq \uu_N$ on $\CC[B_N]$ is implemented via the vector field
\[
\sigma_f := -\sum_i S(f^{(2)}_i) \{f^{(1)}_i, -\}.
\]
\end{prop}

\subsection{Degeneration of $U_q'(\gl_N)$}

It is shown in \cite[Section 12]{CP} that $U_q'(\gl_N)$ is a quantum deformation of $\CC[B_N]$.  To interpret this statement, let $GL_N^*$, the Poisson-Lie group dual to $GL_N$, be given explicitly by
\[
GL_N^* = \left\{(g, f) \mid \text{$g, f \in GL_N$, $g$ lower triangular, $f$ upper triangular, $g_{ii} = f_{ii}^{-1}$}\right\}.
\]
Taking the real form $f^* = g^{-1}$ on $GL_N^*$ yields $\CC[B_N]$ as the corresponding $\star$-Poisson Hopf algebra.  Under this identification, we have the following result of \cite{DKP}.

\begin{theorem}[{\cite[Theorem 7.6 and Remark 7.7(c)]{DKP}}] \label{thm:qdef}
The algebra $U_q'(\gl_N)$ satisfies:
\begin{enumerate}
\item $U_q'(\gl_N)$ is flat over $\CC[q^{\pm 1/2}]$;

\item we have an isomorphism $U_q'(\gl_N) \underset{\CC[q^{\pm 1/2}]} \otimes \CC(q^{1/2}) \simeq U_q(\gl_N)$;

\item $U_q'(\gl_N) / (q^{1/2} - 1) U_q'(\gl_N)$ is commutative;

\item there is an isomorphism of Hopf algebras 
\[
\pi: U_q'(\gl_N) / (q^{1/2} - 1) U_q'(\gl_N) \to \CC[B_N]
\]
which satisfies
\[
\pi\Big((4(q^{1/2} - 1))^{-1}[x, y]\Big) = \{\pi(x), \pi(y)\};
\]

\item $\pi$ takes the special value $\pi(q^{h_i}) = \left(\frac{\det(X_i)}{\det(X_{i-1})}\right)^{1/2}$.
\end{enumerate}
\end{theorem}

\begin{remark}
Note that $(4(q^{1/2} - 1))^{-1}[x, y]$ is a well-defined element of $U_q'(\gl_N)$ by Theorem \ref{thm:qdef}(c).
\end{remark}

For $r$ which is not a root of unity, define $\wU_r(\gl_N)$ to be the corresponding numerical specialization of $U_q'(\gl_N)$.  Denote the specialization map by $\pi_r: U_q'(\gl_N) \to \wU_r(\gl_N)$.  Define also the map of $\CC$-algebras $\varphi: U_q'(\gl_N) \to U_{q}'(\gl_N)$ by 
\begin{equation} \label{eq:auto-def}
\varphi(\oe_i) = \oe_i, \qquad \varphi(\of_i) = \of_i, \qquad \varphi(q^{h_i}) = q^{-h_i}, \qquad \varphi(q) = q^{-1}.
\end{equation}

\begin{theorem} \label{thm:degen}
Fix $z \in U_q'(\gl_N)$.  For sequences of dominant integral signatures $\{\lambda_m\}$ and real quantization parameters $\{q_m\}$ so that $\lim_{m \to \infty} q_m \to 1$ and $\lim_{m \to \infty} -2 \log(q_m) \lambda_m = \lambda$ is dominant regular, we have
\[
\lim_{m \to \infty} (-2 \log(q_m))^{N(N-1)/2} \Tr|_{L_{\lambda_m}}(\pi_{q_m}(z) \cdot q_m^{-2(s, h)}) = \int_{\OO_\Lambda} \pi(\varphi(z)) \prod_{l = 1}^N \left(\frac{\det(X_l)}{\det(X_{l-1})}\right)^{s_l} d\mu_\Lambda,
\]
where we consider $L_{\lambda_m}$ as a representation of $\wU_{q_m}(\gl_N)$ and $\det(X_l)$ as a function on $B_N$ via composition with $\sym: B_N \to P_N^+$ and where $X_l$ is the principal $l \times l$ submatrix of $X \in \OO_\Lambda \subset P^+_N$.
\end{theorem}
\begin{proof}
It suffices to consider monomials $z$, for which we induct on degree.  For the base case, monomials of degree $0$ lie in the Cartan subalgebra, so we have $z = q^{\sum_i 2c_i h_i}$ for some $c_i$.  In this case, we have
\begin{align*}
\lim_{m \to \infty}(-2 &\log(q_m))^{N(N-1)/2} \Tr|_{L_{\lambda_m}}(\pi_{q_m}(z) \cdot q_m^{-2(s, h)}) \\
&=\lim_{m \to \infty} (-2 \log(q_m))^{N(N-1)/2} \Tr|_{L_{\lambda_m}}(e^{-2 \log (q) \sum_i (-c_i + s_i)h_i})\\
&=\lim_{m \to \infty} (-2 \log(q_m))^{N(N-1)/2} \frac{\prod_{i < j} (-c_i + s_i + c_j - s_j)/m}{\prod_{i < j} (e^{\frac{-c_i + s_i + c_j - s_j}{2m}} - e^{\frac{-c_j + s_j + c_i - s_i}{2m}})} \int_{\OO_{\lambda_m + \rho}} e^{-2 \log(q) \sum_i (-c_i + s_i)X_{ii}} d\mu_{\lambda_m + \rho}\\
&=\lim_{m \to \infty} \int_{\OO_{-2 \log(q) (\lambda_m + \rho)}} e^{\sum_i (-c_i + s_i)X_{ii}} d\mu_{-2 \log(q) (\lambda_m + \rho)}\\
&= \int_{\OO_{\Lambda}} \prod_{l = 1}^N \left(\frac{\det(X_l)}{\det(X_{l-1})}\right)^{-c_l + s_l}  d\mu_{\Lambda},
\end{align*}
where the second equality follows from Kirillov's character formula, the third from a change of variables and (\ref{eq:hc-form}), and the last by the Ginzburg-Weinstein isomorphism.  The fact that
\[
\pi\Big(\varphi(q^{\sum_i 2c_i h_i})\Big) = \pi(q^{-\sum_i 2c_i h_i}) = \prod_{l = 1}^N \left(\frac{\det(X_l)}{\det(X_{l-1})}\right)^{-c_l}
\]
by Theorem \ref{thm:qdef} completes the base case.

Suppose that $z = \prod_i \oe_{\beta_i}^{k_i} q^h \prod_i \of_{\beta_i}^{l_i}$ is a PBW monomial of non-zero degree and the claim holds for all monomials of smaller degree.  If all $k_i$ are $0$, not all $l_i$ can be $0$, so the limiting trace is $0$; similarly, $\pi(z)$ is not invariant under the torus action in this case, so the integral is also $0$.  Otherwise, let $i^*$ be minimal so that $k_{i^*} > 0$, and write $z = a b c$ with $a = \oe_{\beta_{i^*}}$, $b = \oe_{\beta_{i^*}}^{k_{i^*} - 1} \prod_{i > i^*} \oe_{\beta_i}^{k_i} q^h$, and $c = \prod_i \of_{\beta_i}^{l_i}$.  We then have that 
\begin{align*}
\Tr|_{L_{\lambda_m}}(\pi_{q_m}(z) q_m^{-2(s,h)}) &= \Tr|_{L_{\lambda_m}}(\pi_{q_m}(bc) q_m^{-2(s,h)} \pi_{q_m}(a)) \\
&= \Tr|_{L_{\lambda_m}}(\pi_{q_m}(bca) q_m^{-2(s, \beta_{i^*})} q_m^{-2(s,h)}) \\
&= q_m^{-2(s, \beta_{i^*})} \Tr|_{L_{\lambda_m}}\Big(\pi_{q_m}(abc + [b,a]c + b[c, a]) q_m^{-2(s, h)}\Big).
\end{align*}
By the relations in Proposition \ref{prop:cp-gr}, we see that 
\[
[b, a] = (q^{f(b, a)} - 1) ab + (\text{terms of lower degree})
\]
for some function $f(b, a)$.  This means that $[b, a] - (q^{f(b, a)} - 1) ab$ lies in a lower degree of the filtration than $ab$.  Solving for the new trace in the rewritten equation
\[
\Tr|_{L_{\lambda_m}}(\pi_{q_m}(z) q_m^{-2(s,h)}) = q_m^{-2(s, \beta_{i^*})} \Tr|_{L_{\lambda_m}}\Big(\pi_{q_m}(q^{f(b, a)}z + ([b,a]c - (q^{f(b, a)} - 1)abc) + b[c, a]) q_m^{-2(s, h)}\Big)
\]
yields the solution
\[
\Tr|_{L_{\lambda_m}}(\pi_{q_m}(z) q_m^{-2(s,h)}) = \frac{q_m^{-(s, \beta_{i^*})} 4(1 - q_m^{1/2})}{1 - q_m^{-2(s, \beta_{i^*}) + f(b, a)}} \Tr|_{L_{\lambda_m}}\Big(\pi_{q_m}\Big(\frac{([b,a]c - (q^{f(b, a)} - 1)abc) + b[c, a]}{4(1 - q^{1/2})}\Big) q_m^{-2(s, h)}\Big).
\]
Using the notation $\pi_a := \pi(\varphi(a))$, $\pi_b := \pi(\varphi(b))$, and $\pi_c := \pi(\varphi(c))$, notice that
\[
\pi\left(\varphi\Big(\frac{([b,a]c - (q^{f(b, a)} - 1)abc) + b[c, a]}{4(1 - q^{1/2})}\Big)\right) = \{\pi_b, \pi_a\}\pi_c + \frac{1}{2}f(b, a) \pi_a \pi_b\pi_c + \pi_b \{\pi_c, \pi_a\}.
\]
Because $([b,a]c - (q^{f(b, a)} - 1)abc) + b[c, a]$ lies in a lower degree of the filtration than $abc$, we conclude by the inductive hypothesis that
\begin{multline} \label{eq:pois-mid}
\lim_{m \to \infty} (-2 \log(q_m))^{N(N-1)/2} \Tr|_{L_{\lambda_m}}(\pi_{q_m}(z) q_m^{-2(s,h)}) \\
= \frac{1}{-(s, \beta_{i^*}) + f(b, a)/2} \int_{\OO_\Lambda} \Big(\{\pi_b, \pi_a\}\pi_c + \frac{1}{2}f(b, a) \pi_a \pi_b\pi_c + \pi_b \{\pi_c, \pi_a\}\Big) \prod_{l = 1}^N \left(\frac{\det(X_l)}{\det(X_{l-1})}\right)^{s_l} d\mu_\Lambda.
\end{multline}
On the other hand, because integrating against Liouville measure kills Poisson brackets and
\[
\left\{\prod_{l = 1}^N \left(\frac{\det(X_l)}{\det(X_{l-1})}\right)^{s_l}, \pi(\oe_{\beta_{i^*}})\right\} = (s, \beta_{i^*}) \prod_{l = 1}^N \frac{\det(X_l)}{\det(X_{l-1})} \pi(\oe_{\beta_{i^*}}),
\]
we have that 
\begin{multline*}
0 = \int_{\OO_\Lambda} \left\{\pi_b \pi_c \prod_{l = 1}^N \left(\frac{\det(X_l)}{\det(X_{l-1})}\right)^{s_l}, \pi_a\right\} d\mu_\Lambda\\
= \int_{\OO_\Lambda} (\{\pi_b, \pi_a\} \pi_c + \pi_b \{\pi_c, \pi_a\} + (s, \beta_{i^*}) \pi_a \pi_b \pi_c) \prod_{l = 1}^N \left(\frac{\det(X_l)}{\det(X_{l-1})}\right)^{s_l} d\mu_\Lambda,
\end{multline*}
which implies that 
\[
\int_{\OO_\Lambda} (\{\pi_b, \pi_a\}\pi_c + \pi_b \{\pi_c, \pi_a\}) \prod_{l = 1}^N \left(\frac{\det(X_l)}{\det(X_{l-1})}\right)^{s_l} d\mu_\Lambda = -\int_{\OO_\Lambda} (s, \beta_{i^*}) \pi_a \pi_b \pi_c \prod_{l = 1}^N \left(\frac{\det(X_l)}{\det(X_{l-1})}\right)^{s_l} d\mu_\Lambda.
\]
Substituting this into (\ref{eq:pois-mid}) completes the induction by yielding the desired
\[
\lim_{m \to \infty} (-2\log(q_m))^{N(N-1)/2} \Tr|_{L_{\lambda_m}}(\pi_{q_m}(z) q_m^{-2(s,h)}) =  \int_{\OO_\Lambda} \pi_a \pi_b \pi_c \prod_{l = 1}^N \left(\frac{\det(X_l)}{\det(X_{l-1})}\right)^{s_l} d\mu_\Lambda. \qedhere
\]
\end{proof}

\subsection{Degenerations of intertwiners}

We now degenerate $\Phi^N_\lambda$ to $F_{k-1}$, for which we wish to represent $\Phi^N_\lambda$ as the evaluation of an element of $U_q'(\gl_N) \otimes W_{k - 1}$ under the map $\ev: U_q'(\gl_N) \to \End_\CC(L_{\lambda + (k - 1)\rho}, L_{\lambda + (k - 1)\rho})$.  Consider the space of invariants $(U_q'(\gl_N) \otimes W_{k - 1})^{U_q'(\gl_N)}$, where the action is given by
\begin{equation} \label{eq:inter-act}
x \cdot (u \otimes w) = \sum x_{(1)} u S(x_{(3)}) \otimes x_{(2)} w
\end{equation}
in the Sweedler notation 
\[
\Delta^{(3)}(x) = \sum x_{(1)} \otimes x_{(2)} \otimes x_{(3)}.
\]
We first show that this space of invariants maps to the space of intertwiners under evaluation.

\begin{lemma} \label{lem:act-inter}
The action of the first tensor factor on $L_{\lambda + (k - 1)\rho}$ sends $(U_q'(\gl_N) \otimes W_{k - 1})^{U_q'(\gl_N)}$ to an intertwiner $L_{\lambda + (k - 1) \rho} \to L_{\lambda + (k - 1)\rho} \otimes W_{k - 1}$.
\end{lemma}
\begin{proof}
Let $z = \sum_i x_i \otimes w_i$ be an element of $(U_q'(\gl_N) \otimes W_{k - 1})^{U_q'(\gl_N)}$. By invariance under $U_q'(\gl_N)$, the action of $q^{h_j}$ satisfies
\[
z = q^{\pm h_j} \cdot z = \sum_i q^{\pm h_j} x_i q^{\mp h_j} \otimes q^{\pm h_j} w_i,
\]
which implies that 
\[
\sum_i x_i q^{\mp h_j} \otimes w_i = \sum_i q^{\pm h_j} x_i \otimes q^{\pm h_j} w_i.
\]
The action of $\oe_j$ satisfies
\begin{multline*}
0 = \oe_j \cdot z = \sum_i \Big(\oe_j x_i q^{(h_j - h_{j + 1})/2} \otimes q^{-(h_j - h_{j + 1})/2} w_i \\+ q^{(h_j - h_{j + 1})/2} x_i q^{(h_j - h_{j + 1})/2} \otimes \oe_j w_i - q^{-1} q^{(h_j - h_{j + 1})/2} x_i \oe_j \otimes q^{(h_j - h_{j + 1})/2} w_i\Big),
\end{multline*}
which upon noting that $q^{h_j} x_i q^{-h_j} \otimes w_i = x_i \otimes q^{-h_j} w_i$ implies that 
\begin{multline*}
\sum_i x_i \oe_j \otimes w_i = \sum_i \Big(q q^{-(h_j - h_{j + 1})/2} \oe_j x_i q^{-(h_j - h_{j + 1})} \otimes q^{-(h_j - h_{j + 1})} w_i + q x_i q^{(h_j - h_{j + 1})/2} \otimes q^{-(h_j - h_{j + 1})/2} e_j w_i\Big)\\
=\sum_i \Big(\oe_j x_i \otimes q^{-(h_j - h_{j + 1})/2} w_i + q^{(h_j - h_{j + 1})/2} x_i \otimes e_j w_i\Big)
= \Delta(\oe_j) z.
\end{multline*}
A similar computation for $\of_j$ yields that $\sum_i x_i \of_j \otimes w_i = \Delta(\of_j)z$, so $z$ gives the desired intertwiner.
\end{proof}

The degeneration $\pi: U_q'(\gl_N) \to \CC[B_N]$ and the automorphism $\varphi$ of (\ref{eq:auto-def}) give rise to a map
\[
(\pi \circ \varphi) \otimes 1: (U_q'(\gl_N) \otimes W_{k-1})^{U_q'(\gl_N)} \to \CC[B_N] \otimes W_{k-1}.
\]
The left dressing action on the first tensor factor gives a $U(\uu_n)$ action on $\CC[B_N] \otimes W_{k-1}$; we now show that $(\pi \circ \varphi) \otimes 1$ lands in the space of invariants for this action.

\begin{lemma} \label{lem:inv-degen}
The image of $(U_q'(\gl_N) \otimes W_{k-1})^{U_q'(\gl_N)}$ under $(\pi \circ \varphi) \otimes 1$ lies in $(\CC[B_N] \otimes W_{k-1})^{U(\uu_N)}$.
\end{lemma}
\begin{proof}
Let $z$ be an element of $(U_q'(\gl_N) \otimes W_{k-1})^{U_q'(\gl_N)}$, and let $z' = ((\pi\circ \varphi) \otimes 1)(z)$.  Write $z = \sum_l x_l \otimes w_l$ and $z' = \sum_l x_l' \otimes w_l$ for $x_l \in U_q'(\gl_N)$, $x_l' = \pi(\varphi(x_l)) \in \CC[B_N]$, and $w_l \in W_{k-1}$.  By invariance, $z$ lies in the zero weight space, so $z'$ lies in the zero weight space of $\CC[B_N] \otimes W_{k-1}$.  By definition of the action of $\oe_j - \of_j \in U_q'(\gl_N)$ on $z$ and the fact that $z$ has weight $0$, we have
\begin{align*}
0 &= \sum_l \Big(\oe_j x_l q^{(h_j - h_{j + 1})/2} \otimes q^{(h_{j+1} - h_j)/2} w_l - q^{-1} q^{(h_j - h_{j + 1})/2} x_l \oe_j \otimes q^{(h_j - h_{j + 1})/2} w_l\Big)\\
&\phantom{=}- \sum_l \Big(\of_j x_l q^{(h_j - h_{j + 1})/2} \otimes q^{(h_{j+1} - h_j)/2} w_l - q q^{(h_j - h_{j + 1})/2} x_l \of_j \otimes q^{(h_j - h_{j + 1})/2} w_l\Big)\\
&\phantom{=}+ \sum_l \Big(q^{(h_j - h_{j + 1})/2} x_l q^{(h_j - h_{j + 1})/2} \otimes \oe_j w_l\Big) - q^{(h_j - h_{j + 1})/2} x_l q^{(h_j - h_{j + 1})/2} \otimes \of_j w_l\Big)\\
&= \sum_l \Big(\oe_j q^{(h_j - h_{j + 1})/2} x_l \otimes w_l - x_l \oe_j q^{(h_{j} - h_{j + 1})/2} \otimes w_l\Big)\\
&\phantom{=}- \sum_l \Big(\of_j q^{(h_j - h_{j + 1})/2} x_l \otimes w_l - x_l \of_j q^{(h_j - h_{j + 1})/2} \otimes w_l\Big)\\
&\phantom{=}+ \sum_l \Big(q^{(h_j - h_{j + 1})/2} x_l q^{(h_j - h_{j + 1})/2} \otimes \oe_j w_l\Big) - q^{(h_j - h_{j + 1})/2} x_l q^{(h_j - h_{j + 1})/2} \otimes \of_j w_l\Big).
\end{align*}
Dividing this equality by $4(q^{1/2} - 1)$, noting that for any $x$ we have 
\[
[\oe_j q^{(h_j - h_{j+1})/2}, x] = \oe_j [q^{(h_j - h_{j + 1})/2}, x] + [\oe_j, x] q^{(h_j - h_{j+1})/2}
\]
and
\[
[\of_j q^{(h_j - h_{j + 1})/2}, x] = \of_j[q^{(h_j - h_{j+1})/2}, x] + [\of_j, x] q^{(h_j - h_{j + 1})/2},
\]
applying $(\pi \circ \varphi) \otimes 1$, and multiplying by $\pi(q^{h_j - h_{j + 1}})$, we find that 
\begin{align*}
0 &= \sum_l \Big(\pi(q^{h_{j} - h_{j+1}}) \pi(\oe_j - \of_j) \{\pi(q^{-(h_j - h_{j+1})/2}), x_l'\} + \pi(q^{(h_j - h_{j + 1})/2}) \{\pi(\oe_j - \of_j), x_l'\}\Big) \otimes w_l \\
&\phantom{=}+ \sum_l x_l' \otimes (E_{j, j + 1} - E_{j + 1, j}) \cdot w_l\\
&= \sum_l \Big(\pi(q^{(h_j - h_{j + 1})/2}) \{\pi(\oe_j - \of_j), x_l'\} -\pi(\oe_j - \of_j) \{\pi(q^{(h_j - h_{j+1})/2}), x_l'\}\Big) \otimes w_l\\
&\phantom{=}+ \sum_l x_l' \otimes (E_{j, j + 1} - E_{j + 1, j}) \cdot w_l,
\end{align*}
where $(E_{j, j + 1} - E_{j + 1, j}) \cdot w_l$ denotes the action of $E_{j, j + 1} - E_{j + 1, j} \in \uu_N$ on $w_l \in W_{k-1}$.  By Proposition \ref{prop:lu-mom}, the dressing action of $d\pi(\oe_j - \of_j)|_e$ is implemented by the vector field
\[
\sigma_{\pi(\oe_j - \of_j)} = - \pi(q^{-(h_j - h_{j+1})/2})\{\pi(\oe_j - \of_j), -\} + \pi(\oe_j - \of_j) \{\pi(q^{-(h_j - h_{j+1})/2}), -\},
\] 
which means that 
\[
\pi(q^{(h_j - h_{j + 1})/2}) \{\pi(\oe_j - \of_j), x_l'\} -\pi(\oe_j - \of_j) \{\pi(q^{(h_j - h_{j+1})/2}), x_l'\} = - \sigma_{\pi(\oe_j - \of_j)}(x_l')
\]
and hence that 
\[
\sum_l \sigma_{\pi(\oe_j - \of_j)}(x_l') \otimes w_l = \sum_l x_l' \otimes (E_{j, j + 1} - E_{j + 1, j}) \cdot w_l.
\]
Under the identification $T_e^*B_N \simeq \uu_N$, we have $d\pi(\oe_j - \of_j)|_e = E_{j, j + 1} - E_{j + 1, j}$ in $\uu_N$ by \cite[Theorem 7.6(b)]{DKP}, so we conclude that $z'$ is invariant under the action of $E_{j, j + 1} - E_{j + 1, j} \in U(\uu_N)$.  A similar argument yields invariance under the action of $i E_{j, j + 1} + i E_{j + 1, j}$, completing the proof.
\end{proof}

\begin{lemma} \label{lem:inter-real}
For any $k$, there exists an element $c_k \in U_q'(\gl_N) \otimes W_{k-1}$ so that 
\[
((\pi \circ \varphi) \otimes 1)(c_k)|_{\OO_\Lambda} = F_{k-1}
\]
and the intertwiner $\Phi_\lambda^N$ is implemented by $c_k|_{L_{\lambda + (k - 1)\rho}}$.
\end{lemma}
\begin{proof}
We show that $(U_q'(\gl_N) \otimes W_{k-1})^{U_q'(\gl_N)}$ is non-zero under the action of (\ref{eq:inter-act}) and then normalize an element of it appropriately.  For this, we first show that it is non-zero under the adjoint action 
\begin{equation} \label{eq:adj-act}
x \cdot (u \otimes w) = \sum x_{(1)} u S(x_{(2)}) \otimes x_{(3)} w.
\end{equation}

\noindent \textit{1. Showing the space of invariants for the action (\ref{eq:adj-act}) is non-zero:} Following \cite{JL}, let $\FF(U_q(\gl_N))$ denote the locally finite part of $U_q(\gl_N)$ under the adjoint action.  By \cite[Theorem 7.4]{JL}, there is an isomorphism
\[
\FF(U_q(\gl_N)) \simeq Z(U_q(\gl_N)) \otimes H_q
\]
for $Z(U_q(\gl_N))$ the center of $U_q(\gl_N)$ and $H_q$ a $U_q(\gl_N)$-submodule of $\FF(U_q(\gl_N))$ under the adjoint action which is a direct sum of $\dim V[0]$ copies of each finite dimensional representation $V$ of $U_q(\gl_N)$.  Because $W_{k-1}^*$ has a one-dimensional zero weight space, there exists an embedding $W_{k-1}^* \to U_q(\gl_N)$ of $U_q(\gl_N)$-representations and therefore a non-zero invariant element under the action of (\ref{eq:adj-act}).

\noindent \textit{2. Showing the space of invariants for the action (\ref{eq:inter-act}) is non-zero:} Let $\cP$ denote the transposition map and $\cR$ the universal $\cR$-matrix of $U_q'(\gl_N)$ and let $\cP_{23}$ and $\cR_{23}$ denote their application in the second and third tensor factor.  Consider the diagram of maps of $U_q'(\gl_N)$-representations
\begin{diagram}
U_q'(\gl_N) \otimes W_{k - 1} \otimes U_q'(\gl_N) & \lTo^{\cP_{23}\cR_{23}} & U_q'(\gl_N) \otimes U_q'(\gl_N) \otimes W_{k - 1} \\
\dTo^{(m_{13} \circ S_3) \otimes 1} & & \dTo_{(m_{12} \circ S_2) \otimes 1} \\
(U_q'(\gl_N) \otimes W_{k - 1})_{\text{(\ref{eq:inter-act})}} & & (U_q'(\gl_N) \otimes W_{k - 1})_{\text{(\ref{eq:adj-act})}}
\end{diagram}
where the $U_q'(\gl_N)$-actions on $U_q'(\gl_N) \otimes W_{k - 1}$ are given by (\ref{eq:inter-act}) and (\ref{eq:adj-act}) as specified.  We claim that $\cP_{23} \cR_{23}$ maps the kernels $K_3$ and $K_2$ of $(m_{13} \circ S_3) \otimes 1$ and $(m_{12} \circ S_2) \otimes 1$ to each other.  Indeed, if $\sum_i u_i \otimes v_i \otimes w_i$ is in $K_2$, then writing $\cR = \sum_j a_j \otimes b_j$, we see that 
\begin{multline*}
((m_{13} \circ S_3) \otimes 1) \cP_{23} \cR_{23} \Big(\sum_i u_i \otimes v_i \otimes w_i\Big)\\ = ((m_{13} \circ S_3) \otimes 1)\Big(\sum_{i, j} u_i \otimes b_j w_i \otimes a_j v_i\Big) = \sum_{i, j} u_i S(v_i) S(a_j) \otimes b_j w_i = 0,
\end{multline*}
where we note that 
\[
((m_{12} \circ S_2) \otimes 1) \Big(\sum_i u_i \otimes v_i \otimes w_i\Big) = \sum_i u_i S(v_i) \otimes w_i = 0.
\]
A similar argument shows that $(\cP_{23} \cR_{23})^{-1}$ maps $K_3$ to $K_2$. Now, we showed that $(U_q'(\gl_N) \otimes W_{k - 1})^{U_q'(\gl_N)}_{(\ref{eq:adj-act})}$ is non-zero.  Choose a one-dimensional space of such invariants and let its preimage under $(m_{12} \circ S_2) \otimes 1$ be $V \subset U_q'(\gl_N) \otimes U_q'(\gl_N) \otimes W_{k - 1}$ so that $V / K_3 \simeq \CC$ as $U_q'(\gl_N)$-representations.  We conclude that $\cP_{23}\cR_{23}(V) / K_2 \simeq \CC$, implying that $(U_q'(\gl_N) \otimes W_{k-1})^{U_q'(\gl_N)}_{(\ref{eq:inter-act})}$ is non-zero.

\noindent \textit{3. Choosing a normalized invariant:} Choose a non-zero element $c_k \in (U_q'(\gl_N) \otimes W_{k-1})^{U_q'(\gl_N)}_{(\ref{eq:inter-act})}$, normalized so that by Lemma \ref{lem:act-inter}, we have $\Phi^N_\lambda = c_k|_{L_{\lambda + (k - 1)\rho}}$. Now, by Lemma \ref{lem:inv-degen}, the image of $c_k$ under $((\pi \circ \varphi) \otimes 1)$ lies in $(\CC[B_N] \otimes W_{k - 1})^{U(\uu_N)}$.  On the other hand, because $\dim W_{k-1}^*[0] = 1$, by \cite[Theorem A]{Ric}, $W_{k-1}^*$ has multiplicity $1$ as a $U(\uu_N)$-representation in $\CC[B_N]$, so $(\CC[B_N] \otimes W_{k - 1})^{U(\uu_N)}$ has dimension $1$.  In particular, this means that $((\pi \circ \varphi) \otimes 1)(c_k)$ restricts to a non-zero multiple of $F_{k - 1}$ in $(\CC[\OO_\Lambda] \otimes W_{k - 1})^{U(\uu_N)}$.  Because the normalization of $c_k$ agrees with that of $\Phi^N_\lambda$, the projection of $c_k$ to $U_q'(\tt_N) \otimes w_{k-1}$ must be $1 \otimes w_{k - 1}$, which implies that the restriction of the $w_{k-1}$-component of $((\pi \circ \varphi) \otimes 1)(c_k)$ to $\CC[T_N]$ is $1$ and hence that 
\[
((\pi \circ \varphi) \otimes 1)(c_k)|_{\OO_\Lambda} = F_{k-1}. \qedhere
\]
\end{proof}

\begin{corr} \label{corr:limit-tr}
For sequences of dominant integral signatures $\{\lambda_m\}$ and real quantization parameters $\{q_m\}$ so that $\lim_{m \to \infty} q_m \to 1$ and $\lim_{m \to \infty} -2 \log(q_m) \lambda_m = \lambda$ is dominant regular, we have
\[
\lim_{m \to \infty} (-2 \log(q_m))^{N(N - 1)/2} \Tr|_{L_{\lambda_m + (k - 1)\rho}}(\pi_{q_m}(\Phi_{\lambda_m}^N) \cdot q_m^{-2(s, h)}) = \int_{\OO_\Lambda} F_{k - 1}(X) \prod_{l = 1}^N \left(\frac{\det(X_l)}{\det(X_{l-1})}\right)^{s_l} d\mu_\Lambda.
\]
\end{corr}
\begin{proof}
This follows by combining Theorem \ref{thm:degen} and Lemma \ref{lem:inter-real}.
\end{proof}

\begin{corr} \label{corr:mac-limit}
For sequences of dominant integral signatures $\{\lambda_m\}$ and real quantization parameters $\{q_m\}$ so that $\lim_{m \to \infty} q_m \to 1$ and $\lim_{m \to \infty} -2 \log(q_m) \lambda_m = \lambda$ is dominant regular, we have
\[
\lim_{m \to \infty} (-2 \log(q_m))^{kN(N - 1)/2} P_{\lambda_m}(q_m^{-2s}; q_m^2, q_m^{2k}) =  \frac{\int_{\OO_\Lambda} F_{k - 1}(X) \prod_{l = 1}^N \left(\frac{\det(X_l)}{\det(X_{l-1})}\right)^{s_l} d\mu_\Lambda}{\prod_{a = 1}^{k-1} \prod_{i < j} (s_i - s_j - a)}.
\]
\end{corr}
\begin{proof}
Set $\lambda_m = m \lambda + (k - 1)\rho$ in Corollary \ref{corr:limit-tr} and explicitly take the limit in Proposition \ref{prop:ek-denom}.
\end{proof}

\subsection{Degeneration of Macdonald operators}

We now put everything together to show that the limiting integral expression satisfies a differential equation in the indices.  This differential equation will be a scaling limit of the difference equations satisfied as a result of the Macdonald symmetry identity, recalled below.  For this, we abuse notation to write $D_{N, q^{2\lambda + 2k\rho}}^r$ for difference operators acting on additive indices $\lambda$ as well as multiplicative variables $q^{2\lambda + 2k \rho}$.  Denote also by $[a]_q$ the $q$-number $[a]_q := \frac{q^a - q^{-a}}{q - q^{-1}}$ and $[a]_{q, l}$ the falling $q$-factorial $[a]_{q, l} := [a]_q \cdots [a - l + 1]_q$.

\begin{prop}[Macdonald symmetry identity] \label{prop:mac-sym}
We have 
\[
P_\lambda(q^{2\mu + 2k \rho}; q^2, q^{2k}) = \prod_{i < j} \frac{[\lambda_i - \lambda_j + k (j - i) + k-1]_{q, k}}{[\mu_i - \mu_j + k (j - i) + k - 1]_{q, k}} P_\mu(q^{2\lambda + 2k \rho}; q^2, q^{2k}).
\]
\end{prop}

\begin{prop} \label{prop:sym-diag}
The operator
\[
\wtilde{D}_{N, q^{2\lambda + 2k\rho}}^r(q^2, q^{2k}) = \prod_{i < j} [\lambda_i - \lambda_j + k (j - i) + k - 1]_{q, k} \circ D_{N, q^{2\lambda + 2k \rho}}^r(q^2, q^{2k}) \circ \prod_{i < j} [\lambda_i - \lambda_j + k (j - i) + k - 1]_{q, k}^{-1}
\]
satisfies
\[
\wtilde{D}_{N, q^{2\lambda + 2k\rho}}^r(q^2, q^{2k}) = \sum_{|I| = r} \prod_{i \in I, j \notin I, i > j} \frac{[\lambda_i - \lambda_j + k (j - i) + k]_q[\lambda_i - \lambda_j + k (j - i) - k + 1]_q}{[\lambda_i - \lambda_j + k (j - i)]_q[\lambda_i - \lambda_j + k (j - i) + 1]_q} T_{q^2, I}
\]
and
\[
\wtilde{D}_{N, q^{2\lambda + 2k\rho}}^r(q^2, q^{2k}) P_\lambda(x; q^2, q^{2k}) = e_r(x) P_\lambda(x; q^2, q^{2k}).
\]
\end{prop}
\begin{proof}
The expression for $\wtilde{D}_{N, q^{2\lambda + 2k\rho}}^r(q^2, q^{2k})$ follows by direct computation, and the eigenvalue identity from the Macdonald symmetry identity. 
\end{proof}

Consider now the operator
\begin{align*}
D_\lambda(q) &= D^1_{N, q^{2\lambda + 2k \rho}}(q^2, q^{2k})^2 - 2 D^2_{N, q^{2\lambda + 2k \rho}}(q^2, q^{2k}) - 2 D^1_{N, q^{2\lambda + 2k \rho}}(q^2, q^{2k}) + N.
\end{align*}
By Proposition \ref{prop:sym-diag}, $D_\lambda(q)$ acts by $\sum_i (x_i - 1)^2$ on 
\[
\prod_{i < j} [\lambda_{m, i} - \lambda_{m, j} + k(j - i) + k - 1]_{q, k}^{-1} P_\lambda(x; q^2, q^{2k}).
\]
We characterize the scaling limit of $D_\lambda(q)$ as a second-order differential operator in the following lemma, whose proof is computational and deferred to Subsection \ref{subsec:mac-lim}

\begin{lemma} \label{lem:mac-op-lim}
Suppose that $\{f_{m}\}$ is a sequence of functions so that for $\{q_m\}$ and $\{\lambda_m\}$ with $\lim_{m \to \infty} q_m = 1$ and $\lim_{m \to \infty} -2 \log(q_m) \lambda_m = \lambda$, we have $\lim_{m \to \infty} f_{m}(\lambda_m; q_m) = f(\lambda)$ for a twice-differentiable function $f$.  Then we have
\[
\lim_{m \to \infty} (-2 \log(q_m))^{-2} D_{\lambda_m}(q_m)f_m(\lambda_m; q_m) = \overline{L}_{p_2}^\text{trig}(k) f(\lambda).
\]
\end{lemma}

Combining Lemma \ref{lem:mac-op-lim} and our results on the degeneration of Macdonald polynomials implies that our representation-valued integrals are diagonalized by the trigonometric Calogero-Moser Hamiltonian.

\begin{theorem} \label{thm:trig-diag}
The trigonometric Calogero-Moser Hamiltonian $\overline{L}_{p_2}^\text{trig}(k)$ is diagonalized on
\[
\frac{1}{\prod_{i < j} (e^{\frac{\lambda_i - \lambda_j}{2}} - e^{-\frac{\lambda_i - \lambda_j}{2}})^k\prod_{a = 1}^{k-1} \prod_{i < j} (s_i - s_j - a)} \int_{\OO_\Lambda} F_{k - 1}(X) \prod_{l = 1}^N \left(\frac{\det(X_l)}{\det(X_{l-1})}\right)^{s_l} d\mu_\Lambda
\]
with eigenvalue $\sum_i s_i^2$.
\end{theorem}
\begin{proof}
Take any sequence $\{q_m\}$ and $\{\lambda_m\}$ so that $\lim_{m \to \infty} q_m = 1$ and $\lim_{m \to \infty} -2 \log(q_m) \lambda_m = \lambda$; for instance, we may take $q_m = e^{-1/2m}$ and $\lambda_m = \lfloor m \lambda \rfloor$.  Notice that 
\[
\lim_{m \to \infty} (2 \log(q_m))^{kN(N-1)/2}\prod_{i < j} [\lambda_{m, i} - \lambda_{m, j} + k(j - i) + k - 1]_{q_m, k} = \prod_{i < j}(e^{\frac{\lambda_i - \lambda_j}{2}} - e^{-\frac{\lambda_i - \lambda_j}{2}})^k
\]
so that by Corollary \ref{corr:mac-limit} we have
\begin{multline*}
\lim_{m \to \infty} (-1)^{kN(N-1)/2}\prod_{i < j} [\lambda_{m, i} - \lambda_{m, j} + k(j - i) + k - 1]_{q_m, k}^{-1} P_{\lambda_m}(q_m^{-2s}; q_m^{2}, q_m^{2k}) \\
=  e^{k(N-1) /2\sum_i \lambda_i} \frac{\int_{\OO_\Lambda} F_{k - 1}(X) \prod_{l = 1}^N \left(\frac{\det(X_l)}{\det(X_{l-1})}\right)^{s_l} d\mu_\Lambda}{\prod_{i < j} (e^{\frac{\lambda_i - \lambda_j}{2}} - e^{-\frac{\lambda_i - \lambda_j}{2}})^k\prod_{a = 1}^{k-1} \prod_{i < j} (s_i - s_j - a)}.
\end{multline*}
Note now that $D_{\lambda_m}(q_m)$ acts by $\sum_i (x_i - 1)^2 = \sum_i (q_m^{-2s_i} - 1)^2$ on 
\[
\prod_{i < j} [\lambda_{m, i} - \lambda_{m, j} + k(j - i) + k - 1]_{q_m, k}^{-1} P_{\lambda_m}(q_m^{-2s}; q_m^{2}, q_m^{2k}),
\]
where $\lim_{m \to \infty} (-2 \log(q_m))^{-2} \sum_i (q_m^{-2s_i} - 1)^2 = \sum_i s_i^2$.  Therefore, by Lemma \ref{lem:mac-op-lim}, we have
\begin{align*}
&\overline{L}_{p_2}^\text{trig}(k) \frac{\int_{\OO_\Lambda} F_{k - 1}(X) \prod_{l = 1}^N \left(\frac{\det(X_l)}{\det(X_{l-1})}\right)^{s_l} d\mu_\Lambda}{\prod_{i < j}  (e^{\frac{\lambda_i - \lambda_j}{2}} - e^{-\frac{\lambda_i - \lambda_j}{2}})^k \prod_{a = 1}^{k-1} \prod_{i < j} (s_i - s_j - a)} \\
&= \lim_{m \to \infty} (-1)^{kN(N-1)/2}(-2 \log(q_m))^{-2} D_{\lambda_m}(q_m)\prod_{i < j} [\lambda_{m, i} - \lambda_{m, j} + k(j - i) + k - 1]_{q_m, k}^{-1} P_{\lambda_m}(q_m^{-2s}; q_m^{2}, q_m^{2k}) \\
&= \Big(\sum_i s_i^2\Big)\frac{\int_{\OO_\Lambda} F_{k - 1}(X) \prod_{l = 1}^N \left(\frac{\det(X_l)}{\det(X_{l-1})}\right)^{s_l} d\mu_\Lambda}{\prod_{i < j} (e^{\frac{\lambda_i - \lambda_j}{2}} - e^{-\frac{\lambda_i - \lambda_j}{2}})^k\prod_{a = 1}^{k-1} \prod_{i < j} (s_i - s_j - a)}. \qedhere
\end{align*}
\end{proof}

\section{The rational case}

\subsection{Statement of the result}

Recall that $f_{k - 1}: \OO_\lambda \to W_{k - 1}$ is the unique $U_N$-equivariant map so that $f_{k-1}(\lambda) = w_{k-1}$.  Define the representation-valued integral
\[
\psi_k(\lambda, s) = \int_{X \in \OO_\lambda} f_{k-1}(X)e^{\sum_{l = 1}^N s_l X_{ll}} d\mu_\lambda
\]
over the coadjoint orbit $\OO_\lambda$.  The integrand and Liouville measure are invariant under the action of the maximal torus of $U_N$, so $\psi_k(\lambda, s)$ lies in $W_{k-1}[0] = \CC \cdot w_{k-1}$.  We interpret the integrals $\psi_k(\lambda, s)$ as complex-valued functions by identifying $\CC \cdot w_{k - 1}$ with $\CC$.  Our first result relates these integrals to the multivariate Bessel functions.

\begin{theorem} \label{thm:rational}
The multivariate Bessel function $\BB_k(\lambda, s)$ admits the integral representation
\[
\BB_k(\lambda, s) = \frac{\Gamma(Nk) \cdots \Gamma(k)}{ \Gamma(k)^{N}\prod_{i < j} (\lambda_i - \lambda_j)^k \prod_{i < j} (s_i - s_j)^{k - 1}} \int_{X \in \OO_\lambda} f_{k - 1}(X) e^{\sum_{l = 1}^N s_l X_{ll}} d\mu_\lambda.
\]
\end{theorem}

\subsection{Adjoints of rational Calogero-Moser operators}

The rational Dunkl operators in variables $\mu_i$ are
\begin{equation} \label{eq:rat-dunkl-def}
D_{\mu_i}(k) := \partial_i + k \sum_{j \neq i} \frac{1}{\mu_i - \mu_j} (1 - s_{ij}),
\end{equation}
where $s_{ij}$ exchanges $\mu_i$ and $\mu_j$. Let $m$ denote the restriction of a differential-difference operator to its differential part.  For a symmetric polynomial $p$, recall that 
\[
m(p(D_{\mu_i}(k))) = \overline{L}_p(k),
\]
for $\overline{L}_p(k)$ defined in (\ref{eq:conj-cm-ham-def}) as a conjugate of the rational Calogero-Moser Hamiltonian corresponding to $p$.  We now compute the adjoint of the differential operator $\oL_p(k)$.
\begin{prop} \label{prop:cm-rat-adj}
For a homogeneous symmetric polynomial $p$, the adjoint of $\oL_p(k)$ as a differential operator is 
\[
\oL_p(k)^\dagger = (-1)^{\deg p} \Delta(\mu)^{2k} \circ \oL_p(k) \circ \Delta(\mu)^{-2k}.
\]
\end{prop}
\begin{proof}
The desired is an equality of symmetric differential operators, so it suffices to verify that for smooth compactly supported symmetric functions $f(\mu)$ and $g(\mu)$ we have 
\[
\int_{\RR^{N-1}} f(\mu) \oL_p(k) g(\mu) d\mu = \int_{\RR^{N-1}} (-1)^{\deg p} [\Delta(\mu)^{2k} \oL_p(k) \Delta(\mu)^{-2k} f(\mu)] g(\mu) d\mu.
\]
In this case, we have $\oL_p(k) = m(p(D_{\mu_i}(k)))$, so it suffices to check that for smooth compactly supported functions $f(\mu)$ and $g(\mu)$ we have
\[
\int_{\RR^{N-1}} f(\mu) D_{\mu_i}(k) g(\mu) d\mu = - \int_{\RR^{N-1}} [\Delta(\mu)^{2k} D_{\mu_i}(k) \Delta(\mu)^{-2k} f(\mu)] g(\mu) d\mu.
\]
For this, notice that 
\begin{align*}
\int_{\RR^{N-1}} f(\mu)& D_{\mu_i}(k) g(\mu) d\mu = \int_{\RR^{N-1}} f(\mu)\Big(\partial_i g(\mu) + k \sum_{j \neq i} \frac{g(\mu) - s_{ij} g(\mu)}{\mu_i - \mu_j}\Big) d\mu\\
&= -\int_{\RR^{N-1}} \partial_i f(\mu) g(\mu) d\mu + k \sum_{j \neq i}\int_{\RR^{N-1}} \frac{1}{\mu_i - \mu_j} f(\mu) g(\mu) d\mu + k \sum_{j \neq i} \int_{\RR^{N-1}}\frac{s_{ij}f(\mu) g(\mu)}{\mu_i - \mu_j} d\mu\\
&= -\int_{\RR^{N-1}} \Big(\partial_i - k \sum_{j \neq i} \frac{1 + s_{ij}}{\mu_i - \mu_j}\Big) f(\mu) g(\mu) d\mu
\end{align*}
and that 
\begin{align*}
\Delta(\mu)^{2k} \circ D_{\mu_i}(k) \circ \Delta(\mu)^{-2k} &= \partial_i - 2k \sum_{j \neq i} \frac{1}{\mu_i - \mu_j} + k \sum_{j \neq i} \frac{1 - s_{ij}}{\mu_i - \mu_j}= \partial_i - k \sum_{j \neq i} \frac{1 + s_{ij}}{\mu_i - \mu_j}. \qedhere
\end{align*}
\end{proof}
We now make precise the adjunction when integrated against a specific domain.  For multi-indices $\alpha = (\alpha_i)$ and $\beta = (\beta_i)$, write $\beta \leq \alpha$ if $\alpha_i \leq \beta_i$ for all $i$.

\begin{prop} \label{prop:rat-formal-adjoint}
Let $A$ be a rectangular domain.  Let $p = \sum_\alpha c_\alpha \mu^\alpha$ be a symmetric homogeneous function.  If for each non-zero monomial $\mu^\alpha$ appearing in $p$, $\partial_\mu^\beta f$ vanishes on the boundary of $A$ for any $\beta \leq \alpha$, then we have the adjunction relation
\[
\int_A (\overline{L}_p(k)f(\mu))\, g(\mu) d\mu = \int_A f(\mu)\, \oL_p(k)^\dagger(g(\mu)) d\mu.
\]
\end{prop}
\begin{proof}
Applying repeated integration by parts with Proposition \ref{prop:cm-rat-adj} shows the two sides of the desired relation differ by the sum of several terms, each of which contains a factor which is the evaluation of $\partial_\mu^\beta(f)$ on a point of the boundary of $A$ for some $\beta \leq \alpha$ with $\mu^\alpha$ appearing in $p$.  These terms vanish, giving the lemma.
\end{proof}

\subsection{A matrix element computation}

Recall that sequences $\{\lambda_i\}_{1 \leq i \leq N}$ and $\{\mu_i\}_{1 \leq i \leq N - 1}$ \textit{interlace} if
\[
\lambda_1 \geq \mu_1 \geq \lambda_2 \geq \cdots \geq \lambda_{N - 1} \geq \mu_{N - 1} \geq \lambda_N,
\]
which we denote by $\mu \prec \lambda$.  Define the real matrix $u(\mu, \lambda)$ by 
\[
u(\mu, \lambda)_{ij} = \begin{cases} \left(\frac{\prod_l (\mu_l - \lambda_j)}{\prod_{l \neq j} (\lambda_l - \lambda_j)}\right)^{1/2} & i = N \\ (\lambda_j - \mu_i)^{-1}\left(\frac{\prod_l (\mu_l - \lambda_j)}{\prod_{l \neq j} (\lambda_l - \lambda_j)}\right)^{1/2} \left(-\frac{\prod_l (\lambda_l - \mu_i)}{\prod_{l \neq i} (\mu_l - \mu_i)}\right)^{1/2} & i < N,
\end{cases}
\]
where each square root is applied to a non-negative real number because $\mu \prec \lambda$, and we take the non-negative branch. The following lemma, whose proof is given in Section \ref{subsec:unit}, shows that $u(\mu, \lambda)$ conjugates a diagonal matrix to a matrix with diagonal principal submatrix.

\begin{lemma} \label{lemma:uconj}
The matrix $u(\mu, \lambda)$ is unitary, and the $(N - 1) \times (N - 1)$ principal submatrix of 
\[
u(\mu, \lambda) \, \diag(\lambda_1, \ldots, \lambda_N) \, u(\mu, \lambda)^*
\]
is $\diag(\mu_1, \ldots, \mu_{N-1})$.
\end{lemma}

We would like to understand a specific matrix element of $u(\mu, \lambda)$ in $W_{k-1}$.  For this, notice that $W_{k-1} \simeq \Sym^{(k-1)N}\CC^N$ as an $SU_N$-representation via an isomorphism sending $w_{k-1}$ to $(x_1\cdots x_N)^{k-1}$.  We now compute an auxiliary quantity.  Let $Z_k(\mu, \lambda)$ denote the coefficient of $(x_1 \cdots x_l)^k$ in the polynomial
\[
\frac{1}{(l - N + 1)!}\prod_{j = 1}^l \left(\sum_{i = 1}^{N-1} \frac{x_i}{\mu_i - \lambda_j} + x_N + \cdots + x_l\right)^k.
\]
By Proposition \ref{prop:matrix-element}, we may express $Z_k(\mu, \lambda)$ via a conjugated Calogero-Moser Hamiltonian, where we recall that $\overline{L}_p(k)$ was defined in (\ref{eq:conj-cm-ham-def}); we defer the proof to Section \ref{subsec:matrix-element-proof}.  The computation of the desired matrix element of $u(\mu, \lambda)$ is an easy consequence.

\begin{prop} \label{prop:matrix-element}
We may write
\[
Z_k(\mu, \lambda) = k!^{-(N - 1)} \Delta(\mu, \lambda)^{-k} \overline{L}_{\mu_{N-1} \cdots \mu_1}(-k)^k \Delta(\mu, \lambda)^k.
\]
\end{prop}

\begin{remark}
It is convenient for us to formulate and prove Proposition \ref{prop:matrix-element} for general $l$.  However, in our main application Lemma \ref{lemma:liouville-int}, it will be only be used with $l = N$.
\end{remark}

\begin{lemma} \label{lemma:liouville-int}
The coefficient of $(x_1\cdots x_{N-1})^{k-1}$ in $u(\mu, \lambda) \cdot (x_1 \cdots x_{N-1})^{k-1}$ is 
\[
(-1)^{(k-1)(N + 2)(N-1)/2}(k - 1)!^{-(N-1)} \Delta(\mu)^{1-k} \Delta(\lambda)^{1 - k} (\overline{L}_{\mu_1\cdots \mu_{N-1}}(1 - k))^{k - 1} \Delta(\mu, \lambda)^{k-1}.
\]
\end{lemma}
\begin{proof}
By Lemma \ref{lemma:uconj}, the desired coefficient is given by 
\[
(-1)^{(k - 1)(N+2)(N-1)/2}\frac{\Delta(\mu, \lambda)^{k-1}}{\Delta(\mu)^{k-1} \Delta(\lambda)^{k-1}} Z_{k-1}(\mu, \lambda),
\]
which is equal to the desired by Proposition \ref{prop:matrix-element}.
\end{proof}

\subsection{Proof of Theorem \ref{thm:rational}}

Integrating over Liouville tori of the Gelfand-Tsetlin integrable system on $\OO_\lambda$ yields the expression
\[
\psi_k(\lambda, s) = \int_{\mu \in \GT_\lambda} \int_{t \in T, X_0 \in \gt^{-1}(\mu)} f_{k-1}(t \cdot X_0) dt\,\, e^{\sum_{l = 1}^N s_l (\sum_i \mu^l_i - \sum_i \mu^{l-1}_i)} \gt_*(d\mu_\lambda),
\]
where $dt$ is an invariant probability measure on $T$ and $\mu^l_i$ are the Gelfand-Tsetlin coordinates.  Recall that $\gt_*(d\mu_\lambda)$ is equal to Lebesgue measure on $\GT_\lambda$ by Proposition \ref{prop:dh-compute}. Adopting the notations of Section \ref{ss:gt-int-sys}, by repeated application of  Lemma \ref{lem:gt-conj} we have for $t = t_1 \cdots t_{N-1}$ in the Gelfand-Tsetlin torus that
\[
t \cdot X_0 = \ad(\overline{v}_1)t_1 \cdot \ad(\overline{v}_2) \cdots  t_{N-1} \cdot\ad(\overline{v}_N) \cdot \lambda.
\]
On the other hand, if $w \in W_{k - 1}$ lies in $\CC[x_1, \ldots, x_l] (x_{l+1} \cdots x_N)^{k - 1}$ under the identification of $W_{k-1} \simeq \Sym^{(k-1)N}\CC^N$ of $SU_N$-representations, then
\[
\int_{T_l} t_l \cdot w\, dt_l = \{\text{coefficient of $(x_1 \cdots x_N)^{k - 1}$ in $w$}\}.
\]
Together, these imply that 
\[
\int_{t \in T, X_0 \in \gt^{-1}(\mu)} f_{k-1}(t \cdot X_0)\, dt = \prod_{m = 1}^{N - 1} W_{m},
\]
where $W_{m}$ denotes the coefficient of $(x_1 \cdots x_{m})^{k-1}$ in $v_m \cdot (x_1 \cdots x_m)^{k-1}$.  Substituting in this result, inducting on $N$, applying the integral formula (\ref{eq:int-rat}), and applying the shift formula
\begin{equation} \label{eq:bessel-shift}
e^{c \sum_i \mu_i} \phi_k(\mu, s) = \phi_k(\mu, s_1 + c, \ldots, s_{N-1} + c)
\end{equation}
for the integral expressions (\ref{eq:int-rat}) in $N - 1$ variables with $c = - s_N$, we obtain
\begin{align*}
\psi_k(\lambda, s) &= \int_{\mu \in \GT_\lambda} \prod_{m = 1}^{N-1} W_m e^{\sum_{l = 1}^N s_l(\sum_i \mu^l - \sum_i \mu^{l-1})} \prod_l d\mu^l_i\\
&= \int_{\mu \prec \lambda} W_{N-1} e^{s_N(\sum_i \lambda - \sum_i \mu)} \!\!\!\! \prod_{1 \leq i < j \leq N - 1} (s_i - s_j)^{k - 1} \phi_k(\mu, s)\prod_i d \mu_i\\
&= e^{s_N \sum_i \lambda_i} \prod_{1 \leq i < j \leq N - 1} (s_i - s_j)^{k - 1} \int_{\mu \prec \lambda} W_{N-1} \phi_k(\mu, s') \prod_i d\mu_i,
\end{align*}
where $s' = (s_1 - s_N, \ldots, s_{N-1} - s_N)$.  Recall now that $v_m$ was chosen so that $(v_m \, \diag(\mu^{m+1}) v_m^*)_{m} = \diag(\mu^m)$, meaning by Lemma \ref{lemma:liouville-int} that 
\[
W_m = \frac{(-1)^{(k-1)(m + 3)(m)/2}(k - 1)!^{-m}}{\Delta(\mu^m)^{k - 1}\Delta(\mu^{m+1})^{k-1}} (\overline{L}_{\mu_1 \cdots \mu_{m}}(1-k))^{k-1} \Delta(\mu^m, \mu^{m+1})^{k-1}.
\]
Substituting this into the previous expression, we find that
\begin{multline*}
\psi_k(\lambda, s) = (-1)^{\frac{(k-1)(N + 2)(N-1)}{2}}e^{s_N \sum_i \lambda_i} \Gamma(k)^{-(N-1)}  \prod_{1 \leq i < j \leq N - 1}\!\!\!\!\! (s_i - s_j)^{k - 1}\Delta(\lambda)^{1 - k} \\ \int_{\mu \prec \lambda}\!\! \Delta(\mu)^{1 - k} \phi_k(\mu, s') \overline{L}_{\mu_1 \cdots \mu_{N-1}}(1-k)^{k-1} \Delta(\mu, \lambda)^{k-1} \prod_i d\mu_i.
\end{multline*}
Recall that $L_p(k) = L_p(1 - k)$ for any symmetric polynomial $p$, which by (\ref{eq:conj-cm-ham-def}) implies that
\begin{equation} \label{eq:cm-shift}
\oL_p(k) = \Delta(\mu)^{1 - 2k} \circ \oL_p(1 - k) \circ \Delta(\mu)^{2k - 1}.
\end{equation}
Recall also that 
\[
\overline{L}_{\mu_1 \cdots \mu_{N-1}}(1-k)^\dagger = (-1)^{N - 1} \Delta(\mu)^{2 - 2k} \circ \oL_{\mu_1\cdots \mu_{N-1}}(1 - k) \circ \Delta(\mu)^{2k - 2}
\]
by Proposition \ref{prop:cm-rat-adj}.  Applying Proposition \ref{prop:rat-formal-adjoint}, we have that
\begin{align*}
\int_{\mu \prec \lambda} &\Delta(\mu)^{1 - k} \phi_k(\mu, s') \overline{L}_{\mu_1 \cdots \mu_{N-1}}(1-k)^{k-1} \Delta(\mu, \lambda)^{k-1} \prod_i d\mu_i \\
&= (-1)^{(N-1)(k-1)}\int_{\mu \prec \lambda} \Delta(\mu, \lambda)^{k-1} \Delta(\mu)^{2 - 2k} \overline{L}_{\mu_1 \cdots \mu_{N-1}}(1-k)^{k-1} \Delta(\mu)^{k - 1} \phi_k(\mu, s') \prod_i d\mu_i \\
&= (-1)^{(N-1)(k-1)}\int_{\mu \prec \lambda} \Delta(\mu, \lambda)^{k-1} \Delta(\mu) \oL_{\mu_1 \cdots \mu_{N-1}}(k)^{k-1} \frac{\phi_k(\mu, s')}{\Delta(\mu)^k} \prod_i d\mu_i \\
&= (-1)^{(N-1)(k-1)}\prod_{i = 1}^{N-1} (s_i - s_N)^{k-1} \int_{\mu \prec \lambda} \frac{\Delta(\mu, \lambda)^{k-1}}{\Delta(\mu)^{k-1}} \phi_k(\mu, s') \prod_i d\mu_i,
\end{align*}
where in the first equality we apply adjunction, in the second equality we apply (\ref{eq:cm-shift}), and in the third equality we apply the inductive hypothesis and (\ref{eq:bessel-ef}). We conclude that 
\begin{align*}
\psi_k(\lambda, s) &=  (-1)^{\frac{(k-1)N(N-1)}{2}} e^{s_N \sum_i \lambda_i} \Gamma(k)^{-(N-1)}  \prod_{1 \leq i < j \leq N }\!\!\!\!\! (s_i - s_j)^{k - 1} \int_{\mu \prec \lambda} \frac{\Delta(\mu, \lambda)^{k-1}}{\Delta(\mu)^{k - 1}\Delta(\lambda)^{k-1}} \phi_k(\mu, s')\prod_i d\mu_i\\
&= \prod_{1 \leq i < j \leq N}(s_i - s_j)^{k - 1} \phi_k(\lambda, s),
\end{align*}
which implies the result by Theorem \ref{thm:bess-bg-int}.

\section{The trigonometric case} \label{sec:trig}
\subsection{Identifying $\Phi_k(\lambda, s)$ with the Heckman-Opdam hypergeometric functions} \label{sec:ho-fn}

In this subsection, we provide details of how to relate the integral formula (\ref{eq:ho-scale}) for $\Phi_k(\lambda, s)$ to the Heckman-Opdam hypergeometric function $\FF_k(\lambda, s)$ for the case where $k > 0$ is a positive integer.  We use the characterization of Theorem \ref{thm:ho-uniq}.  First, we claim the symmetric extension of $\Delta^\trig(\lambda)^{-k} \Phi_k(\lambda, s)$ extends to a holomorphic function of $\lambda$ on a symmetric tubular neighborhood of $\RR^N$.  Observe that (\ref{eq:ho-scale}) has recursive structure
\begin{align} \label{eq:recur}
\frac{\Phi_k(\lambda, s)}{\Delta(e^\lambda)^k} &= \frac{(-1)^{\frac{(k-1)N(N-1)}{2}}}{\Gamma(k)^{N-1}}\int_{\mu \prec \lambda}\!\!\!\!\!\!\!\! e^{s_N (\sum_i \lambda_i - \sum_i \mu_i) - (k-1)\sum_i \mu_i} \frac{\Delta(e^\mu, e^\lambda)^{k-1} \Delta(e^\mu)}{\Delta(e^\lambda)^{2k-1}} \frac{\Phi_k(\mu, s_1, \ldots, s_{N-1})}{\Delta(e^\mu)^k} d\mu.
\end{align}
We induct on $N$ with trivial base case.  For the inductive step, change variables to $\tau_i = \frac{\mu_i - \lambda_{i+1}}{\lambda_i - \lambda_{i+1}}$. We obtain
\begin{multline*}
\frac{\Phi_k(\lambda, s)}{\Delta(e^\lambda)^k} = \frac{(-1)^{\frac{(k-1)N(N-1)}{2}}}{\Gamma(k)^{N-1}} \int_{[0, 1]^{N-1}} \prod_i (\lambda_i - \lambda_{i+1})  e^{s_N (\sum_i \lambda_i - \sum_i \mu_i) - (k-1)\sum_i \mu_i} \\
\frac{\Delta(e^\mu, e^\lambda)^{k-1} \Delta(e^\mu)}{\Delta(e^\lambda)^{2k-1}} \frac{\Phi_k(\mu, s_1, \ldots, s_{N-1})}{\Delta(e^\mu)^k} d\tau,
\end{multline*}
where we view $\mu$ as a function of $\tau$ and $\lambda$ in the integrand.  As a function of $\lambda$, the integrand is meromorphic with poles away from the set $\{\lambda_i \neq \lambda_j\}$.  It is easy to check that there are no poles on the subsets of hyperplanes $\lambda_i = \lambda_{i+1}$ where no other coordinates are equal, so by Hartog's theorem, the integrand is holomorphic in $\lambda$.  By the induction hypothesis, it is also holomorphic in $\tau$, hence the result is holomorphic in $\lambda$ and admits the claimed extension.

We now claim that $\frac{\Phi_k(\lambda, s)}{\Delta^\trig(\lambda)^k}$ satisfies the hypergeometric system; it suffices to show
\[
\oL_p^\trig(k) \frac{\Phi_k(\lambda, s)}{\Delta^\trig(\lambda)^k} = p(s)\frac{\Phi_k(\lambda, s)}{\Delta^\trig(\lambda)^k}
\]
for any symmetric $p$. It was shown in \cite[Proposition 6.3]{BG} that
\[
\oL_{p_2}^\trig(k)\frac{\Phi_k(\lambda, s)}{\Delta^\trig(\lambda)^k} = p_2(s) \frac{\Phi_k(\lambda, s)}{\Delta^\trig(\lambda)^k}
\]
for $\lambda_1 > \cdots > \lambda_N$.  Suppose now that $s$ is not integral and dominant (if $s$ is not dominant, choose $w \in S_N$ so that $ws$ is dominant and replace $s$ by $ws$).  We claim by induction that  $\frac{\Phi_k(\lambda, s)}{\Delta^\trig(\lambda)^k}$ admits a series expansion of the form
\[
\frac{\Phi_k(\lambda, s)}{\Delta^\trig(\lambda)^k} = \prod_{a = 0}^{k - 1} \prod_{i < j} (s_i - s_j + a)^{-1} e^{(s - k \rho, \lambda)} + (\text{l.o.t.}),
\]
where we use $(\text{l.o.t.})$ to denote terms of the form $c_\alpha e^{(s - k \rho - \alpha, \lambda)}$ for $\alpha$ in the positive weight lattice.  By (\ref{eq:recur}), the leading monomial is the same as the leading monomial of
\begin{equation} \label{eq:lead-mon}
\frac{e^{(s_N + \frac{k(N-1)}{2})|\lambda|} e^{(k - 1)(\rho, \lambda) + \frac{(N - 3)(k - 1)}{2}|\lambda| + (k - 1) \lambda_N}}{\Gamma(k)^{N-1} \prod_{a = 0}^{k - 1} \prod_{1\leq i < j \leq N - 1} (s_i - s_j + a) \Delta(e^\lambda)^{2k - 1}} \int_{\mu \prec \lambda} \prod_{i = 1}^{N-1} (e^{\lambda_i} - e^{\mu_i})^{k-1} e^{(\mu, s')} \prod_i d\mu_i,
\end{equation}
where $s' = (s_1 - s_N, \ldots, s_{N-1} - s_N)$.  Note that for $b_2 < b_1$, the expression
\begin{align*}
\int_{b_2}^{b_1} (e^{b_1} - e^x)^{k-1} e^{xs} dx &= \sum_{l = 0}^{k-1} (-1)^l \int_{b_2}^{b_1} \binom{k - 1}{l} e^{x(s + l)} e^{b_1(k - 1 - l)} dx \\
&= e^{b_1(k - 1 + s)} \sum_{l = 0}^{k - 1} \frac{(-1)^l}{s + l} \binom{k - 1}{l} (1 - e^{(b_2 - b_1)(s + l)})
\end{align*}
has leading monomial $e^{b_1(k - 1 + s)}$ in $b$ with coefficient $\sum_{l = 0}^{k - 1} \frac{(-1)^l}{s + l} \binom{k - 1}{l} = \frac{\Gamma(k)}{\prod_{a = 0}^{k - 1} (s + a)}$. Therefore, the leading monomial in (\ref{eq:lead-mon}) is 
\[
\frac{e^{(s_N + \frac{k(N-1)}{2})|\lambda|} e^{(k - 1)(\rho, \lambda) + \frac{(N - 3)(k - 1)}{2}|\lambda| + (k - 1)\lambda_N} e^{(\lambda, s') + (k - 1)(|\lambda| - \lambda_N)}}{\prod_{1 \leq i < j \leq N} \prod_{a = 0}^{k - 1} (s_i - s_j + a)e^{((2k - 1)\rho, \lambda) + \frac{(2k - 1)(N-1)}{2}|\lambda|}} = \prod_{1 \leq i < j \leq N} \prod_{a = 0}^{k - 1} (s_i - s_j + a)^{-1} e^{(\lambda, s - k \rho)},
\]
as claimed.  By the results of \cite[Section 4.2]{HS}, for such $s$, any analytic symmetric eigenfunction of $\oL_{p_2}^\trig(k)$ with leading monomial $e^{(s - k \rho, \lambda)}$ diagonalizes $\oL_{p}^\trig(k)$ for any $p$.  Therefore, $\frac{\Phi_k(\lambda, s)}{\Delta^\trig(e^\lambda)^k}$ satisfies the full hypergeometric system, and our computation of its leading monomial coefficient implies that
\[
\FF_k(\lambda, s) = \frac{\Gamma(Nk) \cdots \Gamma(k)}{\Gamma(k)^N} \frac{\Phi_k(\lambda, s)}{\Delta(e^\lambda)^k}
\]
for non-integral $s$.  Both sides of the expression are holomorphic functions of $s$, so this continues to hold for non-generic $s$, yielding Theorem \ref{thm:bg-int}.

\subsection{Statement of the result}

Let $F_{k - 1}: \OO_\Lambda \to W_{k - 1}$ be the unique $U_N$-equivariant map so that $F_{k - 1}(\Lambda) = w_{k - 1}$.  Define the representation-valued integral
\[
\Psi_k(\lambda, s) = \int_{X \in \OO_\Lambda} F_{k-1}(X) \prod_{l = 1}^N \left(\frac{\det(X_l)}{\det(X_{l - 1})}\right)^{s_l} d\mu_\Lambda,
\]
where $X_l$ denotes the principal $l \times l$ submatrix of $X$.  As in the rational case, the integrand and Liouville measure in the definition of $\Psi_k(\lambda, s)$ are invariant under the action of the maximal torus of $U_N$, so $\Psi_k(\lambda, s)$ lies in $W_{k - 1}[0] = \CC \cdot w_{k - 1}$.  We will again interpret it as a complex-valued function via the identification of $\CC \cdot w_{k - 1}$ with $\CC$. Our result in the trigonometric setting uses these integrals to express the Heckman-Opdam hypergeometric functions.

\begin{theorem} \label{thm:trig-integral}
The Heckman-Opdam hypergeometric function $\FF_k(\lambda, s)$ admits the integral representation
\[
\FF_k(\lambda, s) = \frac{\Gamma(Nk) \cdots \Gamma(k)}{\Gamma(k)^{N} \prod_{i < j} (e^{\frac{\lambda_i - \lambda_j}{2}} - e^{-\frac{\lambda_i - \lambda_j}{2}})^k \prod_{a = 1}^{k - 1}\prod_{i < j}(s_i - s_j - a)} \int_{X \in \OO_\Lambda} F_{k - 1}(X) \prod_{l = 1}^N \left(\frac{\det(X_l)}{\det(X_{l - 1})}\right)^{s_l} d\mu_\Lambda,
\]
where $X_l$ is the principal $l \times l$ submatrix of $X$.
\end{theorem}

\subsection{Adjoints of trigonometric Calogero-Moser operators}

The trigonometric Dunkl operators in variables $\mu_i$ are defined by 
\[
T_{\mu_i}(k) := \partial_i + k \sum_{\alpha > 0} (\alpha, \mu_i) \frac{1}{1 - e^{-\alpha}} (1 - s_\alpha) - k (\rho, \mu_i).
\]
For a symmetric polynomial $p$, $m(p(T_{\mu_i}(k))) = \overline{L}^\trig_p(k)$ is the conjugate (\ref{eq:conj-tcm-ham-def}) of the trigonometric Calogero-Moser Hamiltonian corresponding to $p$.  Let $w_0 \in S_{N - 1}$ be the permutation sending $i$ to $N - i$.  We now compute the adjoint of the differential operator $\oLt_p(k)$ via an analogue of \cite[Lemma 7.8]{Opd3}.

\begin{prop} \label{prop:cm-trig-adj}
For a homogeneous symmetric polynomial $p$, the adjoint of $\oLt_p(k)$ as a differential operator is 
\[
\oLt_p(k)^\dagger = (-1)^{\deg p} \Delta^\trig(\mu)^{2k} \circ \oLt_{p}(k) \circ \Delta^\trig(\mu)^{-2k}.
\]
\end{prop}
\begin{proof}
As in the proof of Proposition \ref{prop:cm-rat-adj}, it suffices to check that for smooth compactly supported functions $f(\mu)$ and $g(\mu)$ we have 
\[
\int_{\RR^{N-1}} f(\mu) T_{\mu_i}(k) g(\mu) d\mu = - \int_{\RR^{N-1}} [\Delta^\trig(\mu)^{2k} \circ w_0 \circ T_{\mu_{N - i}}(k) \circ w_0 \circ \Delta^\trig(\mu)^{-2k} f(\mu)] g(\mu) d\mu.
\]
For this, we compare both sides by computing
\begin{align*}
\int_{\RR^{N-1}}& f(\mu) T_{\mu_i}(k) g(\mu) d\mu = \int_{\RR^{N-1}} f(\mu)\Big(\partial_i - k \sum_{j < i} \frac{1 - s_{ij}}{1 - e^{\mu_i - \mu_j}} + k \sum_{j > i} \frac{1 - s_{ij}}{1 - e^{\mu_j - \mu_i}} - k \frac{N - 2i}{2}\Big) g(\mu) d\mu\\
&= - \int_{\RR^{N-1}} \partial_i f(\mu) g(\mu) d\mu - k \sum_{j < i} \int_{\RR^{N-1}}\Big(\frac{1}{1 - e^{\mu_i - \mu_j}} - \frac{s_{ij}}{1 - e^{\mu_j - \mu_i}}\Big) f(\mu) g(\mu) d\mu \\
&\phantom{====} + k \sum_{j > i} \int_{\RR^{N-1}}\Big(\frac{1}{1 - e^{\mu_j - \mu_i}} - \frac{s_{ij}}{1 - e^{\mu_i - \mu_j}}\Big) f(\mu) g(\mu) d\mu - k \frac{N - 2i}{2} \int_{\RR^{N - 1}} f(\mu) g(\mu) d\mu\\
&= -\int_{\RR^{N-1}} \Big(\partial_i - k \sum_{j < i}\frac{s_{ij} + e^{\mu_j - \mu_i}}{1 - e^{\mu_j - \mu_i}} + k \sum_{j > i}\frac{s_{ij} + e^{\mu_i -\mu_j}}{1 - e^{\mu_i - \mu_j}} + k  \frac{N - 2i}{2}\Big) f(\mu) g(\mu) d\mu
\end{align*}
and 
\begin{align*}
\Delta^\trig(\mu)^{2k} \circ w_0 \circ &T_{\mu_{N - i}}(k) \circ w_0 \circ \Delta^\trig(\mu)^{-2k}\\
&= \Delta^\trig(\mu)^{2k} \circ \Big(\partial_i - k \sum_{j > i} \frac{1 - s_{ij}}{1 - e^{\mu_i - \mu_j}} + k \sum_{j < i} \frac{1 - s_{ij}}{1 - e^{\mu_j - \mu_i}} + k \frac{N - 2i}{2}\Big) \circ \Delta^\trig(\mu)^{-2k}\\
&= \partial_i - k \sum_{j > i} \frac{1 - s_{ij}}{1 - e^{\mu_i - \mu_j}} + k \sum_{j < i} \frac{1 - s_{ij}}{1 - e^{\mu_j - \mu_i}} + k \frac{N - 2i}{2} + k \sum_{j \neq i} \frac{1 + e^{\mu_i - \mu_j}}{1 - e^{\mu_i - \mu_j}}\\
&= \partial_i + k \sum_{j > i} \frac{e^{\mu_i - \mu_j} + s_{ij}}{1 - e^{\mu_i - \mu_j}} - k \sum_{j < i} \frac{e^{\mu_j - \mu_i} + s_{ij}}{1 - e^{\mu_j - \mu_i}} + k \frac{N - 2i}{2}. \qedhere
\end{align*}
\end{proof}

We again make precise the adjunction when integrated against a specific domain.

\begin{prop} \label{prop:trig-formal-adjoint}
Let $A$ be a rectangular domain.  Let $p = \sum_\alpha c_\alpha \mu^\alpha$ be a symmetric function.  If for each non-zero monomial $\mu^\alpha$ appearing in $p$, $\partial_\mu^\beta f$ vanishes on the boundary of $A$ for any $\beta \leq \alpha$, then we have the adjunction relation
\[
\int_A (\overline{L}^\trig_p(k)f(\mu))\, g(\mu) d\mu = \int_A f(\mu)\, \oLt_p(k)^\dagger(g(\mu)) d\mu.
\]
\end{prop}
\begin{proof}
The proof is the same as for Proposition \ref{prop:rat-formal-adjoint}.
\end{proof}

\subsection{Matrix elements in the trigonometric case}

Take $l \geq N - 1$ and consider variables $\lambda_1, \ldots, \lambda_l$ and $\mu_1, \ldots, \mu_{N-1}$.  Recall that $Z_k(e^\mu, e^\lambda)$ denotes the coefficient of $(x_1 \cdots x_l)^k$ in the polynomial
\[
\frac{1}{(l - N + 1)!} \prod_{j = 1}^l \left(\sum_{i = 1}^{N-1} \frac{x_i}{e^{\mu_i} - e^{\lambda_j}} + x_N + \cdots + x_l\right)^k.
\]
We express $Z_k(e^\mu, e^\lambda)$ via trigonometric Calogero-Moser Hamiltonians in Proposition \ref{prop:trig-matrix-element}.

\begin{prop} \label{prop:trig-matrix-element}
We have the identity
\[
Z_k(e^\mu, e^\lambda) = k!^{-(N-1)}\Delta(e^\mu, e^\lambda)^{-k}\left(e^{-\sum_i \mu_i} \overline{L}^\trig_{\prod_{i = 1}^{N - 1} \left(\mu_i - k \frac{N - 2}{2}\right)}(-k)\right)^k \Delta(e^\mu, e^\lambda)^k.
\]
\end{prop}
\begin{proof}
We use the result in the rational case.  By Proposition \ref{prop:matrix-element} and (\ref{eq:conj-tcm-ham-def}), we have that 
\[
Z_k(e^\mu, e^\lambda) = k!^{-(N-1)}\Delta(e^\mu, e^\lambda)^{-k} \Big(D_{e^{\mu_{N-1}}}(-k) \cdots D_{e^{\mu_{1}}}(-k)\Big)^k \Delta(e^\mu, e^\lambda)^k.
\]
It therefore suffices to check that
\[
D_{e^{\mu_{N-1}}}(-k) \cdots D_{e^{\mu_{1}}}(-k) = e^{-\sum_i \mu_i} \Big(T_{\mu_{N-1}}(-k) - k \frac{N-2}{2}\Big) \cdots \Big(T_{\mu_{1}}(-k) - k \frac{N - 2}{2}\Big)
\]
on $\CC[e^{\mu_i}]^{S_{N-1}}$.  We may rewrite $T_{\mu_i}(-k)$ in the form 
\begin{equation} \label{eq:trig-rat}
T_{\mu_i}(-k) = \partial_{\mu_i} - k \sum_{j \neq i} \frac{e^{\mu_i}}{e^{\mu_i} - e^{\mu_j}} (1 - s_{ij}) - k \sum_{j < i} s_{ij} + k \frac{N - 2}{2} = e^{\mu_i} D_{e^{\mu_i}}(-k) - k \sum_{j < i} s_{ij} + k \frac{N-2}{2},
\end{equation}
where $D_{e^{\mu_i}}(-k)$ is the rational Dunkl operator in the exponential variables $e^{\mu_i}$, which implies that
\begin{equation} \label{eq:dunkl-trig-to-rat}
 D_{e^{\mu_i}}(-k) = e^{-\mu_i} \Big(T_{\mu_i}(-k) + k \sum_{j < i} s_{ij} - k \frac{N-2}{2}\Big).
\end{equation}
Further, we may check that for $i > j$, we have $(T_{\mu_i}(-k) + k s_{ij}) e^{-\mu_j} = e^{-\mu_j} T_{\mu_i}(-k)$, so substituting (\ref{eq:dunkl-trig-to-rat}) and shifting each $e^{-\mu_i}$ term to the beginning of the expression, we see that
\[
D_{e^{\mu_{N-1}}}(-k) \cdots D_{e^{\mu_1}}(-k) = e^{-\sum_i \mu_i} \prod_{i = 1}^{N-1} \Big(T_{\mu_i}(-k) - k \frac{N - 2}{2}\Big). \qedhere
\]
\end{proof}

\subsection{Proof of Theorem \ref{thm:trig-integral}}

We again compute $\Psi_k(\lambda, s)$ by integrating over the Liouville tori given by the Gelfand-Tsetlin coordinates.  We may write
\begin{equation} \label{eq:dressing-red}
\Psi_k(\lambda, s) = \int_{\mu \in \GT_\lambda} \int_{t \in T, X_0 \in \GT^{-1}(\mu)} F_{k-1}(t \cdot X_0) dt\, e^{\sum_{l = 1}^N s_l(\sum_i \mu^l_i - \sum_i \mu^{l-1}_i)} \GT_*(d\mu_\Lambda),
\end{equation}
where $dt$ is the invariant probability measure on the torus, and $\mu^l_i$ are the logarithmic Gelfand-Tsetlin coordinates. As in the rational case, by Lemma \ref{lem:gt-conj}, we have
\[
\int_{t \in T, X_0 \in \GT^{-1}(\mu)} F_{k-1}(t \cdot X_0) dt = \prod_{m = 1}^{N - 1} W_{m},
\]
where $W_{m}$ denotes the coefficient of $(x_1 \cdots x_{m})^{k-1}$ in $v_m \cdot (x_1 \cdots x_m)^{k-1}$. Noting that $\GT_*(d\mu_\Lambda) = 1_{\GT_\lambda} \cdot dx$ by Proposition \ref{prop:dh-compute} and inducting on $N$, we transform (\ref{eq:dressing-red}) to
\begin{align*}
\Psi_k(\lambda, s) &= \int_{\mu \in \GT_\lambda} \prod_{m = 1}^{N-1} W_m e^{\sum_{l = 1}^N s_l (\sum_i \mu^l_i - \sum_i \mu^{l - 1}_i)} \prod_{i} d \mu^l_i \\
&=\int_{\mu \prec \lambda} W_{N - 1}  e^{s_N(\sum_i \lambda_i - \sum_i \mu_i)}  \prod_{a = 1}^{k - 1} \prod_{1\leq i < j \leq N - 1} (s_i - s_j - a) \Phi_k(\mu, s) \prod_i d\mu_i\\
&= \prod_{1 \leq i < j \leq N - 1} (s_i - s_j - a) e^{s_N \sum_i \lambda_i} \int_{\mu \prec \lambda} W_{N - 1} \Phi_k(\mu, s') \prod_i d\mu_i,
\end{align*}
where $s' = (s_1 - s_N, \ldots, s_{N-1} - s_N)$ and the last equality follows from the $c = - s_N$ case of the shift identity
\begin{equation} \label{eq:trig-shift}
e^{c \sum_i \lambda_i} \Phi_k(\lambda, s) = \Phi_k(\lambda, s_1 + c, \ldots, s_N + c).
\end{equation}
Notice that $(v_m \, \diag(e^{\mu^{m+1}}) v_m^*)_{m} = \diag(e^{\mu^m})$, so writing $\mu := \mu^{N-1}$ and $\lambda := \mu^N$, we have by Lemma \ref{lemma:uconj} and Proposition \ref{prop:trig-matrix-element} that
\begin{align*}
W_{N - 1} &= (-1)^{(k-1)(N+2)(N-1)/2}\frac{\Delta(e^{\mu}, e^{\lambda})^{k-1}}{\Delta(e^{\mu})^{k-1} \Delta(e^{\lambda})^{k-1}} Z_{k-1}(e^{\mu}, e^{\lambda})\\
&= \frac{(-1)^{(k-1)(N+2)(N-1)/2} \Gamma(k)^{-(N-1)}}{\Delta(e^{\mu})^{k-1} \Delta(e^{\lambda})^{k-1}} \left(e^{-\sum_i \mu_i} \overline{L}^\trig_{\prod_{i = 1}^{N - 1} \left(\mu_i - (k - 1) \frac{N - 2}{2}\right)}(1-k)\right)^{k - 1} \Delta(e^\mu, e^\lambda)^{k - 1}\\
&= \frac{(-1)^{(k-1)(N+2)(N-1)/2} \Gamma(k)^{-(N-1)}}{\Delta^\trig(\mu)^{k-1} \Delta(e^{\lambda})^{k-1}} \left(e^{-\sum_i \mu_i} \overline{L}^\trig_{\mu_1 \cdots \mu_{N - 1}}(1-k)\right)^{k - 1}  e^{- \frac{(k - 1)(N - 1)}{2} |\mu|} \Delta(e^\mu, e^\lambda)^{k - 1},
\end{align*}
where in the last line we use that $e^{- c |\mu|} T_{\mu_i}(k) e^{c|\mu|} = T_{\mu_i}(k) + c$.  We conclude that 
\begin{multline*}
\Psi_k(\lambda, s) = (-1)^{(k-1)(N+2)(N-1)/2} \Gamma(k)^{-(N-1)} \prod_{1 \leq i < j \leq N - 1} (s_i - s_j - a) e^{s_N \sum_i \lambda_i} \Delta(e^{\lambda})^{1 - k}\\
\int_{\mu \prec \lambda} \frac{\Phi_k(\mu, s')}{\Delta^\trig(\mu)^{k-1} } \left(e^{-\sum_i \mu_i} \overline{L}^\trig_{\mu_1 \cdots \mu_{N - 1}}(1-k)\right)^{k - 1}  e^{- \frac{(k - 1)(N - 1)}{2} |\mu|} \Delta(e^\mu, e^\lambda)^{k - 1} \prod_i d\mu_i.
\end{multline*}
Recall that $L_p^\trig(k) = L_p^\trig(1 - k)$ for any symmetric polynomial $p$, which by (\ref{eq:conj-tcm-ham-def}) implies that 
\begin{equation} \label{eq:cm-shift-trig}
\oL_p^\trig(k) = \Delta^\trig(\mu)^{1 - 2k} \circ \oL_p^\trig(1 - k) \circ \Delta^\trig(\mu)^{2k - 1}.
\end{equation}
In addition, by Proposition \ref{prop:cm-trig-adj}, we have that 
\begin{equation} \label{eq:cm-adj-trig}
\oL^\trig_{\mu_1 \cdots \mu_{N - 1}}(1 - k)^\dagger = (-1)^{N - 1} \Delta^\trig(\mu)^{2 - 2k} \oL^\trig_{\mu_1 \cdots \mu_{N - 1}}(k - 1) \circ \Delta^\trig(\mu)^{2k - 2}.
\end{equation}
Applying Proposition \ref{prop:trig-formal-adjoint}, we have that 
\begin{align*}
\int_{\mu \prec \lambda}& \frac{\Phi_k(\mu, s')}{\Delta^\trig(\mu)^{k-1}} \Big(e^{-\sum_i \mu_i} \overline{L}^\trig_{\mu_1 \cdots \mu_{N - 1}}(1-k)\Big)^{k - 1} e^{-\frac{(k - 1)(N - 1)}{2} |\mu|} \Delta(e^\mu, e^\lambda)^{k - 1} \prod_i d\mu_i\\
&= (-1)^{(N - 1)(k-1)}\int_{\mu \prec \lambda} e^{-\frac{(k - 1)(N - 1)}{2} |\mu|} \Delta(e^\mu, e^\lambda)^{k - 1} \Delta^\trig(\mu)^{2 - 2k}\\
&\phantom{==================} \Big(\overline{L}^\trig_{\mu_1 \cdots \mu_{N - 1}}(1-k) e^{-\sum_i \mu_i}\Big)^{k - 1} \Delta^\trig(\mu)^{k - 1}\Phi_k(\mu, s')\prod_i d\mu_i\\
&= (-1)^{(N-1)(k-1)}\int_{\mu \prec \lambda} e^{-\frac{(k - 1)(N - 1)}{2} |\mu|} \Delta(e^\mu, e^\lambda)^{k - 1} \Delta^\trig(\mu) \Big(\overline{L}^\trig_{\mu_1 \cdots \mu_{N - 1}}(k) e^{-\sum_i \mu_i}\Big)^{k - 1} \frac{\Phi_k(\mu, s')}{\Delta^\trig(\mu)^k}\prod_i d\mu_i\\
&= (-1)^{(N-1)(k-1)} \prod_{a = 1}^{k-1} \prod_{i = 1}^{N - 1}(s_i - s_N - a) \int_{\mu \prec \lambda} \frac{\Delta(e^\mu, e^\lambda)^{k - 1}}{\Delta(e^\mu)^{k-1}} \Phi_k(\mu, s') e^{-(k - 1)|\mu|} \prod_i d\mu_i,
\end{align*}
where we apply adjunction and (\ref{eq:cm-adj-trig}) in the first equality, (\ref{eq:cm-shift-trig}) in the second equality, and the inductive hypothesis and (\ref{eq:trig-shift}) in the last equality.  Substituting into our previous expression, we obtain
\begin{align*}
\Psi_k(\lambda, s) &= (-1)^{(k-1)N(N-1)/2} \Gamma(k)^{-(N-1)} \prod_{a = 1}^{k-1}\prod_{1 \leq i < j \leq N} (s_i - s_j - a) \\
&\phantom{=================}\int_{\mu \prec \lambda} \frac{\Delta(e^\mu, e^\lambda)^{k - 1}}{\Delta(e^\mu)^{k-1} \Delta(e^\lambda)^{k-1}}e^{s_N(|\lambda| - |\mu|)} \Phi_k(\mu, s)e^{-(k-1)|\mu|}\prod_i d\mu_i\\
&= \prod_{a = 1}^{k-1} \prod_{1 \leq i < j \leq N} (s_i - s_j - a) \Phi_k(\lambda, s),
\end{align*}
which implies the theorem by normalizing via Theorem \ref{thm:bg-int}.

\section{Proofs of some technical lemmas} \label{sec:tech}

\subsection{Proof of Lemma \ref{lem:mac-op-lim}} \label{subsec:mac-lim}

For a subset $I$ of indices, denote by $1_I$ and $2_I$ the vectors with $1$ and $2$ in the indices of $I$ and $0$ elsewhere.  We first expand the Macdonald difference operators in $\log(q_m)$, yielding
\begin{align*}
&D^r_{N, q_m^{2\lambda_m + 2k\rho}}(q_m^2, q_m^{2k})f_m(\lambda_m; q_m)\\
&\phantom{===}= q_m^{2r(r - n)k} \sum_{|I| = r} \prod_{i \in I, j \notin I} \frac{q_m^{k} q_m^{2(\lambda_{m, i} - \lambda_{m, j} + k(j - i))} - q_m^{-k}}{q_m^{2(\lambda_{m, i} - \lambda_{m, j} + k(j - i))} - 1} f_m(\lambda_m + 1_I; q_m)\\
&\phantom{===}= \sum_{|I| = r} \prod_{i \in I, j \notin I} \left(1 + (1 - q_m^{k}) \frac{q_m^{2(\lambda_{m, i} - \lambda_{m, j} + k(j - i))} + q_m^{-k}}{1 - q_m^{2(\lambda_{m, i} - \lambda_{m, j} + k(j - i))}}\right) f_m(\lambda_m + 1_I; q_m)\\
&\phantom{===}= \sum_{|I| = r} \left(1 + \sum_{i \in I, j \notin I} (1 - q_m^k) \frac{q_m^{2(\lambda_{m, i} - \lambda_{m, j} + k(j - i))} + q_m^{-k}}{1 - q_m^{2(\lambda_{m, i} - \lambda_{m, j} + k(j - i))}} + C_r(\lambda_m, q_m) \log(q_m)^2\right) f_m(\lambda_m + 1_I; q_m) + O(\log(q_m)^3)
\end{align*}
for some functions $C_r(\lambda_m, q_m) = o(\log(q_m)^{-1})$.  Specializing this, we see that 
\begin{multline*}
D^1_{N, q_m^{2\lambda_m + 2k\rho}}(q_m^{2}, q_m^{2k})f_m(\lambda_m; q_m)\\
 = \sum_{i = 1}^N \left(1 + \sum_{j \neq i} (1 - q_m^k) \frac{q_m^{2(\lambda_{m, i} - \lambda_{m, j} + k(j - i))} + q_m^{-k}}{1 - q_m^{2(\lambda_{m, i} - \lambda_{m, j} + k(j - i))}} + C_1(\lambda_m, q_m)\log(q_m)^{2} \right) f_m(\lambda_m + 1_i; q_m) + O(\log(q_m)^{3})
\end{multline*}
and
\begin{align*}
&D^1_{N, q_m^{2\lambda_m + 2k\rho}}(q_m^2, q_m^{2k})^2f_m(\lambda_m; q_m) \\
&= \sum_{i = 1}^N \left(1 + S_1(\lambda_m, q_m)\log(q_m)^{2}\right) f_m(\lambda_m + 2_i; q_m) + \sum_{i_1 \neq i_2} \left(1 + S_2(\lambda_m, q_m) \log(q_m)^{2}\right) f_m(\lambda_m + 1_{i_1, i_2}; q_m) + O(\log(q_m)^{2}) \\
&\phantom{=} + (1 - q_m^{k}) \sum_{i = 1}^N \sum_{j \neq i} \Big(\frac{q_m^{2(\lambda_{m, i} - \lambda_{m, j} + k(j - i))} + q_m^{-k}}{1 - q_m^{2(\lambda_{m, i} - \lambda_{m, j} + k(j - i))}} + \frac{q_m^{2(\lambda_{m, i} + 1 - \lambda_{m, j} + k(j - i))} + q_m^{-k}}{1 - q_m^{2(\lambda_{m, i} + 1 - \lambda_{m, j} + k(j - i))}}\Big) f_m(\lambda_m + 2_i; q_m) \\
&\phantom{=} + (1 - q_m^{k})\sum_{i_1 \neq i_2}\sum_{j \neq i_1, i_2} \Big(\frac{q_m^{2(\lambda_{m, i_2} - \lambda_{m, j} + k(j - i_2))} + q_m^{-k}}{1 - q_m^{2(\lambda_{m, i_2} - \lambda_{m, j} + k(j - i_2))}} +\frac{q_m^{2(\lambda_{m, i_1} - \lambda_{m, j} + k(j - i_1))} + q_m^{-k}}{1 - q_m^{2(\lambda_{m, i_1} - \lambda_{m, j_1} + k(j_1 - i_1))}}\Big)f_m(\lambda_m + 1_{i_1, i_2}; q_m) \\
&\phantom{=} + (1 - q_m^{k})\sum_{i_1 \neq i_2}\Big( \frac{q_m^{2(\lambda_{m, i_2} - \lambda_{m, i_1} + k(i_1 - i_2))} + q_m^{-k}}{1 - q_m^{2(\lambda_{m, i_2} - \lambda_{m, i_1} + k(i_1 - i_2))}} + \frac{q_m^{2(\lambda_{m, i_1} - \lambda_{m, i_2} - 1 + k(i_2 - i_1))} + q_m^{-k}}{1 - q_m^{2(\lambda_{m, i_1} - \lambda_{m, i_2} - 1 + k(i_2 - i_1))}}\Big) f_m(\lambda_m + 1_{i_1, i_2}; q_m)
\end{align*}
for some functions $S_1(\lambda_m, q_m)$ and $S_2(\lambda_m, q_m)$, both of which are $o(\log(q_m)^{-1})$.  We define
\begin{align*}
A_{i_1, i_2}(\lambda_m, q_m) &= \frac{1}{1 - q_m^{2}}\Big(\frac{q_m^{2(\lambda_{m, i_2} - \lambda_{m, i_1} + k(i_1 - i_2))} + q_m^{-k}}{1 - q_m^{2(\lambda_{m, i_2} - \lambda_{m, i_1} + k(i_1 - i_2))}} + \frac{q_m^{2(\lambda_{m, i_1} - \lambda_{m, i_2} - 1 + k(i_2 - i_1))} + q_m^{-k}}{1 - q_m^{2(\lambda_{m, i_1} - \lambda_{m, i_2} - 1 + k(i_2 - i_1))}}\Big) \\
B_{i, j}(\lambda_m, q_m) &= \frac{q_m^{2(\lambda_{m, i} - \lambda_{m, j} + k(j - i))} + q_m^{-k}}{1 - q_m^{2(\lambda_{m, i} - \lambda_{m, j} + k(j - i))}} + \frac{q_m^{2(\lambda_{m, i} + 1 - \lambda_{m, j} + k(j - i))} + q_m^{-k}}{1 - q_m^{2(\lambda_{m, i} + 1 - \lambda_{m, j} + k(j - i))}}
\end{align*}
so that 
\begin{multline*}
\sum_{i_1 \neq i_2}\Big( \frac{q_m^{2(\lambda_{m, i_2} - \lambda_{m, i_1} + k(i_1 - i_2))} + q_m^{-k}}{1 - q_m^{2(\lambda_{m, i_2} - \lambda_{m, i_1} + k(i_1 - i_2))}} + \frac{q_m^{2(\lambda_{m, i_1} - \lambda_{m, i_2} - 1 + k(i_2 - i_1))} + q_m^{-k}}{1 - q_m^{2(\lambda_{m, i_1} - \lambda_{m, i_2} - 1 + k(i_2 - i_1))}}\Big) f_m(\lambda_m + 1_{i_1, i_2}; q_m)\\
 = (1 - q_m^2) \sum_{i_1 \neq i_2} A_{i_1, i_2}(\lambda_m, q_m) f_m(\lambda_m + 1_{i_1, i_2}; q_m) + O(\log(q_m)^{2})
\end{multline*}
and
\begin{multline*}
\sum_{j \neq i} \Big(\frac{q_m^{2(\lambda_{m, i} - \lambda_{m, j} + k(j - i))} + q_m^{-k}}{1 - q_m^{2(\lambda_{m, i} - \lambda_{m, j} + k(j - i))}} + \frac{q_m^{2(\lambda_{m, i} + 1 - \lambda_{m, j} + k(j - i))} + q_m^{-k}}{1 - q_m^{2(\lambda_{m, i} + 1 - \lambda_{m, j} + k(j - i))}}\Big) f_m(\lambda_m + 2_i; q_m)\\
 = \sum_{j \neq i} B_{i, j}(\lambda_m, q_m) f_m(\lambda_m + 2_i; q_m),
\end{multline*}
Notice that
\[
\lim_{m \to \infty} A_{i_1, i_2}(\lambda_m, q_m) = \frac{k e^{2\lambda_{ i_1} - 2 \lambda_{ i_2}} - 2(k - 2) e^{\lambda_{ i_1} - \lambda_{ i_2}} + k}{(1 - e^{\lambda_{ i_1} - \lambda_{ i_2}})^2} \text{ and }
\lim_{m \to \infty} B_{i, j}(\lambda_m, q_m) = - \frac{2(1 + e^{\lambda_{ i} - \lambda_{ j}})}{1 - e^{\lambda_{ i} - \lambda_{ j}}}.
\]
We have also that 
\begin{align*}
D^2_{N, q_m^{2\lambda_m + 2k\rho}}&(q_m^{2}, q_m^{2k})f_m(\lambda_m; q_m)\\
&\phantom{===}= \sum_{i_1 \neq i_2} \left(1 + C_2(\lambda_m, q_m) \log(q_m)^{2}\right) f_m(\lambda_m + 1_{i_1, i_2}; q_m) + O(\log(q_m)^{2}) \\
&\phantom{======} + (1 - q_m^{k})\sum_{i_1 \neq i_2} \sum_{j \neq i_1, i_2} \Big(\frac{q_m^{2(\lambda_{m, i_1} - \lambda_{m, j} + k(j - i_1))} + q_m^{-k}}{1 - q_m^{2(\lambda_{m, i_1} - \lambda_{m, j} + k(j - i_1))}}\Big) f_m(\lambda_m + 1_{i_1,i_2}; q_m)\\
&\phantom{======} + (1 - q_m^{k})\sum_{i_1 \neq i_2} \sum_{j \neq i_1, i_2} \Big(\frac{q_m^{2(\lambda_{m, i_2} - \lambda_{m, j} + k(j - i_2))} + q_m^{-k}}{1 - q_m^{2(\lambda_{m, i_2} - \lambda_{m, j} + k(j - i_2))}}\Big) f_m(\lambda_m + 1_{i_1,i_2}; q_m).
\end{align*}
Together, these imply that 
\begin{align*}
D_{\lambda_m}(q_m) f_m(\lambda_m; q_m) &= \sum_{i = 1}^N \left(1 + (1 - q_m^{k}) \sum_{j \neq i} B_{i, j}(\lambda_m, q_m) + S_1(\lambda_m, q_m)\log(q_m)^{2}\right) f_m(\lambda_m + 2_i; q_m)\\
&\phantom{==}-2\sum_{i = 1}^N \left(1 + \sum_{j \neq i} (1 - q_m^{k}) \frac{q_m^{2(\lambda_{m, i} - \lambda_{m, j} + k(j - i))} + q_m^{-k}}{1 - q_m^{2(\lambda_{m, i} - \lambda_{m, j} + k(j - i))}}\right) f_m(\lambda_m + 1_i; q_m)\\
&\phantom{==} + (1 - q_m^{k})(1 - q_m^{2}) \sum_{i_1\neq i_2} A_{i_1, i_2}(\lambda_m, q_m)f_m(\lambda_m + 1_{i_1, i_2}; q_m) \\
&\phantom{==}+ (C_2(\lambda_m, q_m) - S_2(\lambda_m, q_m))\log(q_m)^{2} f_m(\lambda_m + 1_{i_1, i_2}; q_m) + N f_m(\lambda_m; q_m) + O(\log(q_m)^{2}).
\end{align*}
Taking limits in the previous expression yields that 
\begin{align*}
\lim_{m \to \infty} (-2\log(q_m))^{-2} D_{\lambda_m}(q_m) f_m(\lambda_m; q_m) &= \Delta f(\lambda) - k \sum_{i \neq j} \frac{1 + e^{\lambda_i - \lambda_j}}{1 - e^{\lambda_i - \lambda_j}} \partial_if(\lambda) + R(\lambda) f(\lambda)\\
&= \Big(\Delta - k \sum_{i < j} \frac{1 + e^{\lambda_i - \lambda_j}}{1 - e^{\lambda_i - \lambda_j}} (\partial_i - \partial_j)+ R(\lambda)\Big) f(\lambda)
\end{align*}
for some function $R(\lambda)$.  Note that $f_m(\lambda_m) \equiv 1$ is the Macdonald polynomial in $q^{2\lambda_m}$ corresponding to the empty partition, hence we conclude that
\begin{align*}
D_\lambda(q) \cdot 1 &= p_2(q^{2k\rho}) - 2 p_1(q^{2k\rho}) + N = \sum_i (q^{2k\rho_i} - 1)^2,
\end{align*}
which implies that 
\[
\lim_{m \to \infty} (-2 \log(q_m))^{-2} D_{\lambda_m}(q_m) \cdot 1 =  k^2(\rho, \rho),
\]
hence $R(\lambda) \equiv k^2 (\rho, \rho)$. We conclude that
\[
\lim_{m \to \infty} (-2 \log(q_m))^{-2} D_{\lambda_m}(q_m) f_m(\lambda_m; q_m) = \Big(\Delta - k \sum_{i < j} \frac{1 + e^{\lambda_i - \lambda_j}}{1 - e^{\lambda_i - \lambda_j}} (\partial_i - \partial_j)+ k^2(\rho, \rho)\Big) f(\lambda),
\]
where we note that by \cite[Theorem 2.1.1]{HS}, we have 
\[
\overline{L}_{p_2}^\text{trig}(k) = \Delta - k \sum_{i < j} \frac{1 + e^{\lambda_i - \lambda_j}}{1 - e^{\lambda_i - \lambda_j}} (\partial_i - \partial_j)+ k^2(\rho, \rho). \qedhere
\]

\subsection{Proof of Lemma \ref{lemma:uconj}} \label{subsec:unit}

We verify the statement by direct computation.  Write $u = u(\mu, \lambda)$ and $\lambda = \diag(\lambda_1, \ldots, \lambda_N)$.  Define the non-negative real numbers $x_1, \ldots, x_{N-1}$ by 
\[
x_i^2 = - \frac{\prod_{j} (\lambda_j - \mu_i)}{\prod_{j \neq i} (\mu_j - \mu_i)},
\]
where we note that the right side of the definition is non-negative because $\lambda$ and $\mu$ interlace.  Define $y = \sum_i \lambda_i - \sum_i \mu_i$.  For $i < N$, our definition of $u$ implies that
\begin{equation} \label{eq:unit-rel}
u_{ij} = \frac{x_i}{\lambda_j - \mu_i} u_{Nj}.
\end{equation}
We first claim that $u\lambda = \mu' u$ for the matrix 
\[
\mu' = \left(\begin{array}{ccccc|c}
 \mu_1& & & & & x_1  \\ 
 & \mu_2  & & & & x_2  \\
& &\ddots & & & \vdots \\
& & & \mu_{N-2} & & x_{N-2}  \\
& & & &\mu_{N-1}  &x_{N-1} \\ \hline
x_1 & x_2 & \cdots & x_{N-2} &x_{N-1} & y \end{array}\right).
\]
For $i < N$, this holds for each element of row $i$ by the equality
\[
\lambda_j u_{ij} = \mu_j u_{ij} + x_i u_{Nj}
\]
implied by (\ref{eq:unit-rel}).  For row $N$, we must check that 
\[
\lambda_j u_{Nj} = \sum_{i = 1}^{N-1} x_i u_{ij} + y u_{Nj} = \left(y + \sum_{i = 1}^{N-1} \frac{x_i^2}{\lambda_j - \mu_i}\right) u_{Nj},
\]
for which it suffices to check that 
\begin{equation} \label{eq:bottom-row-check}
\sum_{i = 1}^{N-1} \frac{\prod_{l \neq j} (\lambda_l - \mu_i)}{\prod_{l \neq i} (\mu_l - \mu_i)}  = \sum_{i \neq j} \lambda_i - \sum_i \mu_i.
\end{equation}
The left side of (\ref{eq:bottom-row-check}) is a symmetric rational function in the $\mu_i$ which may be expressed as a quotient
\[
\frac{P(\mu)}{\prod_{i < j} (\mu_i - \mu_j)},
\]
whose numerator $P(\mu)$ has degree at most $\frac{N(N-1)}{2} + 1$ in the $\mu$-variables.  Therefore, $P(\mu)$ is antisymmetric, meaning the quotient is symmetric of degree at most $1$.  In particular, it takes the form $C_1 + C_2\sum_i \mu_i$ for $C_1$ and $C_2$ constant in $\mu$.  Noting that the coefficient of $\mu_1^{N - 1}\mu_2^{N - 3} \mu_3^{N - 4} \cdots \mu_{N - 2}$ in $P(\mu)$ is $-1$ shows that $C_2 = -1$.  Finally, $C_1$ is a polynomial of degree $1$ in $\lambda$, so it is given by
\[
C_1 = \sum_i \frac{\mu_i^{N - 2} (-1)^{N-2} \sum_{l \neq j} \lambda_l}{\prod_{l \neq i} (\mu_l - \mu_i)} = \sum_i \frac{\mu_i^{N - 2}}{\prod_{l \neq i} (\mu_i - \mu_l)} \cdot \left( \sum_{l \neq j} \lambda_l\right) = \sum_{i \neq j} \lambda_i,
\]
where the last equality follows by noting that $\sum_i \frac{\mu_i^{N - 2}}{\prod_{l \neq i} (\mu_i - \mu_l)}$ is symmetric of degree $0$ in $\mu$ and a rational function whose denominator is $\prod_{i < j} (\mu_i - \mu_j)$ and whose numerator contains $\mu_1^{N- 2} \mu_2^{N - 3} \cdots \mu_{N-2}$ with coefficient $1$.  This establishes (\ref{eq:bottom-row-check}).

It remains to check that $u$ is unitary.  For this, we check that the columns of $u$ are orthonormal.  Choose any $1 \leq a < b \leq N$.  We have that 
\[
\sum_i u_{ia}u_{ib} = \left(\sum_i \frac{x_i^2}{(\lambda_a - \mu_i)(\lambda_b - \mu_i)} + 1\right) u_{Na}u_{Nb} = \left(1 - \sum_i \frac{\prod_{j \neq a, b} (\lambda_j - \mu_i)}{\prod_{j \neq i} (\mu_j - \mu_i)}\right) u_{Na}u_{Nb}.
\]
Observe that $\sum_i \frac{\prod_{j \neq a, b} (\lambda_j - \mu_i)}{\prod_{j \neq i} (\mu_j - \mu_i)}$ is symmetric in the $\mu_i$ and may be expressed as a rational function with denominator $\prod_{i < j} (\mu_i - \mu_j)$ and numerator of degree at most $\frac{N(N-1)}{2}$ in $\mu$.  Further, the coefficient of $\mu_1^{N - 2} \mu_2^{N-3} \cdots \mu_{N-2}$ in the numerator is $1$, so we conclude that 
\begin{equation} \label{eq:unit-off-diag}
1 - \sum_i \frac{\prod_{j \neq a, b} (\lambda_j - \mu_i)}{\prod_{j \neq i} (\mu_j - \mu_i)} = 0,
\end{equation}
hence $\sum_i u_{ia}u_{ib} = 0$.  It remains only to show that 
\[
1 = \sum_i u_{ia}^2 = \left(1 + \sum_i \frac{x_i^2}{(\lambda_a - \mu_i)^2}\right) u_{Na}^2,
\]
for which we must check that 
\[
\frac{\prod_{l \neq a} (\lambda_l - \lambda_a)}{\prod_l (\mu_l - \lambda_a)} = 1 - \sum_i \frac{\prod_{j \neq a} (\lambda_j - \mu_i)}{(\lambda_a - \mu_i) \prod_{j \neq i}(\mu_j - \mu_i)},
\]
which is equivalent to 
\begin{equation} \label{eq:diag-unit}
\prod_{l \neq a} (\lambda_l - \lambda_a) = \prod_l (\mu_l - \lambda_a) \left(1 - \sum_i \frac{\prod_{j \neq a} (\lambda_j - \mu_i)}{(\lambda_a - \mu_i) \prod_{j \neq i}(\mu_j - \mu_i)}\right).
\end{equation}
View both sides of (\ref{eq:diag-unit}) as polynomials in $\lambda_a$.  If $\lambda_a = \lambda_b$ for $b \neq a$, the right side becomes
\[
1 - \sum_i \frac{\prod_{j \neq a, b} (\lambda_j - \mu_i)}{\prod_{j \neq i}(\mu_j - \mu_i)} = 0
\]
by (\ref{eq:unit-off-diag}).  Therefore, both sides of (\ref{eq:diag-unit}) are polynomials in $\lambda_a$ of the same degree with the same roots and the same leading coefficient $(-1)^{N-1}$, so they are equal, completing the proof.

\begin{remark}
The expressions above for $x_i^2$ and $y$ appeared previously in \cite{Ner}.  Similar computations appeared also in \cite{GK, FoRa}.
\end{remark}

\subsection{Proof of Proposition \ref{prop:matrix-element}} \label{subsec:matrix-element-proof}

Before beginning the proof, we outline our approach.  We first obtain an alternate expression for $Z_1(\mu, \lambda)$ in Lemma \ref{lemma:k1-matrix-element}.  We then observe that $Z_k(\mu, \lambda)$ is a constant multiple of $Z_1(\mu', \lambda')$ for sets of variables $\mu'$ and $\lambda'$ which contain $k$ duplicate copies of each value of $\mu$ and $\lambda$.  Relating Calogero-Moser Hamiltonians at different values of $k$ in Lemma \ref{lemma:cm-eq} leads to the result.  Recall here that $D_{\mu_i}(\kappa)$ denotes the rational Dunkl operator of (\ref{eq:rat-dunkl-def}).

\begin{lemma} \label{lemma:k1-matrix-element}
For any $\kappa \in \CC$, we have 
\[
\Delta(\mu, \lambda)^{-\kappa} D_{\mu_{N-1}}(- \kappa) \cdots D_{\mu_{1}}(- \kappa) \Delta(\mu, \lambda)^\kappa = \kappa^{N-1} Z_1(\mu, \lambda).
\]
\end{lemma}
\begin{proof}
We first claim that
\begin{equation} \label{eq:k1-ind}
\Delta(\mu, \lambda)^{-\kappa} D_{\mu_{a}}(- \kappa) \cdots D_{\mu_{1}}(- \kappa) \Delta(\mu, \lambda)^\kappa = \kappa^a \sum_{\substack{\sigma: \{1, \ldots, a\} \\ \phantom{===}\to \{1, \ldots, l\} \\ \sigma(i) \neq \sigma(j)}} \prod_{i = 1}^a (\mu_i - \lambda_{\sigma(i)})^{-1}.
\end{equation}
Taking $a = N - 1$ in (\ref{eq:k1-ind}) and expanding the product in the definition of $Z_1(\mu, \lambda)$ then completes the proof.  We prove (\ref{eq:k1-ind}) by induction on $a$.  The base case $a = 1$ holds because $D_{\mu_1}(- \kappa)$ acts by $\partial_1$ on the symmetric function $\Delta(\mu, \lambda)^\kappa$ in $\mu$.  For the inductive step, note that $D_{\mu_{a}}(- \kappa) \cdots D_{\mu_{1}}(- \kappa) \Delta(\mu, \lambda)^\kappa$ is symmetric in $\mu_{a + 1}, \ldots \mu_{N-1}$ by the inductive hypothesis.  Applying $D_{\mu_{a+1}}(-\kappa)$, we see that
\begin{multline*}
\Delta(\mu, \lambda)^\kappa D_{\mu_{a + 1}}(- \kappa) (D_{\mu_{a}}(- \kappa) \cdots D_{\mu_{1}}(- \kappa) \Delta(\mu, \lambda)^\kappa)\\ = \kappa^{a + 1} \sum_{j = 1}^l (\mu_{a + 1} - \lambda_j)^{-1}\!\!\!\!\!\!\!\!\!\!\! \sum_{\substack{\sigma: \{1, \ldots, a\}\\ \phantom{===} \to \{1, \ldots, l\} \\ \sigma(i) \neq \sigma(j)}}\!\!\!\!\!\! \prod_{i = 1}^a (\mu_i - \lambda_{\sigma(i)})^{-1} - \kappa^{a + 1} \!\!\!\!\!\! \sum_{\substack{\sigma: \{1, \ldots, a\}\\ \phantom{===} \to \{1, \ldots, l\} \\ \sigma(i) \neq \sigma(j)}}\!\!\!\!\!\! \prod_{i = 1}^a (\mu_i - \lambda_{\sigma(i)})^{-1} \sum_{i = 1}^a (\mu_{a + 1} - \lambda_{\sigma(i)})^{-1}\\
= \kappa^{a+1} \sum_{\substack{\sigma: \{1, \ldots, a + 1\} \\ \phantom{===}\to \{1, \ldots, l\} \\ \sigma(i) \neq \sigma(j)}} \prod_{i = 1}^{a+1} (\mu_i - \lambda_{\sigma(i)})^{-1},
\end{multline*}
where we repeatedly make use of the identity
\[
\frac{1}{\mu_{a + 1} - \mu_i} \Big((\mu_{a + 1} - \lambda_j) - (\mu_i - \lambda_j)\Big) = 1. \qedhere
\]
\end{proof}

\begin{proof}[Proof of Proposition \ref{prop:matrix-element}]
Replace $l$ by $kl$ and apply Lemma \ref{lemma:k1-matrix-element} with $\kappa = \frac{1}{k}$, $k$ copies of each $\lambda_j$, and $k(N - 1)$ different variables $\mu_1^1, \ldots, \mu_1^k, \ldots, \mu_{N - 1}^1, \ldots, \mu_{N-1}^k$.  We obtain
\begin{equation} \label{eq:me-inter}
\Delta(\{\mu_i^j\}, \{\lambda_i\})^{-1} D_{\mu_{N-1}^k}(-1/k) \cdots D_{\mu_1^1}(-1/k) \Delta(\{\mu_i^j\}, \{\lambda_i\}) = k^{-(N - 1)k} Z_1(\{\mu_i^j\}, \{\lambda_i^j\}).
\end{equation}
Now, make the specialization $\mu_1^1 = \cdots = \mu_1^k = \mu_1, \ldots, \mu_{N-1}^1 = \cdots = \mu_{N-1}^k = \mu_{N-1}$.  We first claim that
\[
Z_1(\{\mu_i^j\}, \{\lambda_i^j\}) = k!^{N-1} Z_k(\{\mu_i\}, \{\lambda_i\})
\]
under this specialization.  Indeed, we see that 
\begin{align*}
Z_1(\{\mu_i^j\}, \{\lambda_i^j\})& = \sum_{\substack{\sigma: \{1, \ldots, (N - 1)\} \times \{1, \ldots, k\} \\ \phantom{===} \to \{1, \ldots, l\} \times \{1, \ldots, k\} \\ \sigma(i_1, j_1) \neq \sigma(i_2, j_2)}} \prod_{i, j} (\mu_i^j - \lambda_{\sigma(i, j)_1}^{\sigma(i, j)_2})^{-1} \\
&= \sum_{\substack{\sigma^1, \ldots, \sigma^{N-1} \subset \{1, \ldots, l\} \times \{1, \ldots, k\} \\ |\sigma^i| = k \\ \sigma^i \cap \sigma^j = \emptyset}} k!^{N-1} \prod_i \prod_{(j, p) \in \sigma^i} (\mu_i - \lambda_j^p)^{-1}\\
&=k!^{N-1} \sum_{\substack{\sigma^1_1, \ldots, \sigma^1_l, \ldots, \sigma^{N-1}_1, \ldots, \sigma^{N-1}_l \\ \sum_j \sigma^i_j = k \\  \sum_i \sigma^i_j \leq k}} \prod_{i} \prod_j \binom{k}{\sigma^1_j, \ldots, \sigma^{N-1}_j} (\mu_i - \lambda_j)^{-\sigma_j^i},
\end{align*}
which is a direct expansion of $Z_k(\{\mu_i\}, \{\lambda_i\})$.  The conclusion will now follow from Lemma \ref{lemma:cm-eq}, which describes what occurs under specialization to the other side of Lemma \ref{lemma:k1-matrix-element}.  Indeed, applying Lemma \ref{lemma:cm-eq} for $p(y) = y_1^1 \cdots y_{N-1}^k$ to (\ref{eq:me-inter}), we see that 
\begin{align*}
Z_k(\{\mu_i\}, \{\lambda_i\}) &= k!^{-(N-1)} k^{(N-1)k} k^{- (N - 1)k} \Delta(\{\mu_i\}, \{\lambda_i\})^{-k} D_{\mu_{N-1}}(-k)^k \cdots D_{\mu_1}(-k)^k \Delta(\{\mu_i\}, \{\lambda_i\})^k\\
& = k!^{-(N-1)} \Delta(\{\mu_i\}, \{\lambda_i\})^{-k} D_{\mu_{N-1}}(-k)^k \cdots D_{\mu_1}(-k)^k \Delta(\{\mu_i\}, \{\lambda_i\})^k. \qedhere
\end{align*}
\end{proof}

\begin{lemma} \label{lemma:cm-eq}
Let $p \in \CC[y_1^1, \ldots, y_{N-1}^k]^{S_{k(N-1)}}$ be a symmetric polynomial. Then the map $\Res_k: \CC[\mu_i^j] \to \CC[\mu_i]$ given by $\mu_i^j \mapsto \mu_i$ satisfies 
\begin{multline*}
\Res_k \circ p(D_{\mu_1^1}(-k^{-1}), \ldots, D_{\mu_{N-1}^k}(-k^{-1}))\\
 = p\Big(\frac{1}{k} D_{\mu_1}(-k), \ldots, \frac{1}{k} D_{\mu_1}(-k), \ldots, \frac{1}{k} D_{\mu_{N-1}}(-k), \ldots, \frac{1}{k} D_{\mu_{N-1}}(-k)\Big) \circ \Res_k.
\end{multline*}
\end{lemma}
\begin{proof}
For any $c$ and $n$, let $H_{c, n}$ denote the rational Cherednik algebra associated to $S_n$ with parameter $c$, given in terms of generators and relations by 
\[
H_{c, n} := \left\langle x_1, \ldots, x_n, y_1, \ldots, y_n \mid [x_i, x_j] = [y_i, y_j] = 0, [y_i, x_i] = \delta_{ij} - c \sum_{j \neq i} s_{ij}, [y_i, x_j] = c s_{ij}\right\rangle.
\]
Let $H_{1/k, (N - 1)k}$ and $H_{k, N - 1}$ denote the rational Cherednik algebras of $S_{(N - 1)k}$ and $S_{N-1}$, respectively.  Within $H_{1/k, (N - 1)k}$ and $H_{k, N - 1}$, denote the power sums $p_a(x) = \sum_{i, j} (x_i^j)^a$ and $p_a'(x) = \sum_i x_i^a$, and define $p_a(y), p_a'(y)$ similarly.  Write $\Theta_{1/k, (N - 1)k}: H_{1/k, (N - 1)k} \to \End(\CC[\mu_i^j])$ and $\Theta_{k, N - 1}: H_{k, N - 1}\to \End(\CC[\mu_i])$ for the Dunkl embeddings induced by $\Theta_{1/k, (N - 1)k}(x_i^j) = \mu_i^j$, $\Theta_{1/k, (N - 1)k}(y_i^j) = D_{\mu_i^j}(-1/k)$, $\Theta_{k, N - 1}(x_i) = k x_i$, and $\Theta_{k, N - 1}(y_i) = \frac{1}{k} D_{\mu_i}(-k)$.  In this language, we wish to show that
\begin{equation} \label{eq:desired}
\Res_k \circ \Theta_{1/k, (N - 1)k}(p_a(y)) = \Theta_{k, N - 1}(p_a'(y)) \circ \Res_k.
\end{equation}
Suppose first that the statement held for $p_2(y)$. Then, we have for any $a$ that
\begin{equation} \label{eq:partial-ad}
\Res_k \circ \Theta_{1/k, (N - 1)k} (\ad_{p_2(y)}^a p_a(x)) = \Theta_{k, N - 1} (\ad_{p_2'(y)}^a p_a'(x)) \circ \Res_k
\end{equation}
Recall that for $h = \frac{1}{2} \sum_{i, j} (x_{i, j} y_{i, j} + y_{i, j} x_{i, j})$ and $h' = \frac{1}{2} \sum_i (x_i y_i + y_i x_i)$, the triples
\[
(f, e, h) = \Big(\frac{1}{2}p_2(y), - \frac{1}{2}p_2(x), h\Big) \qquad \text{ and } \qquad (f', e', h') = \Big(\frac{1}{2}p_2'(y), -\frac{1}{2} p_2'(x), h'\Big)
\]
are copies of $\sl_2$ inside $H_{1/k, (N - 1)k}$ and $H_{k, N - 1}$ corresponding to the $SL_2(\CC)$-actions given by 
\[
\left(\begin{matrix} a & b \\ c & d \end{matrix}\right) x_i = a x_i + b y_i, \qquad \left(\begin{matrix} a & b \\ c & d \end{matrix}\right) y_i = c x_i + d y_i,
\]
and similar formulas for $x_i^j$, $y_i^j$.  In particular, $p_a(x)$ and $p_a'(x)$ are highest weight vectors of weight $a$ for these representations, so $\ad_{p_2(y)/2}^a p_a(x)$ and $\ad_{p_2'(y)/2}^a p_a'(x)$ are the same fixed constant multiple of
\[
\left(\begin{matrix} 0 & 1 \\ -1 & 0 \end{matrix}\right)p_a(x) = p_a(y) \text{ and } \left(\begin{matrix} 0 & 1 \\ -1 & 0 \end{matrix}\right)p_a'(x) = p_a'(y),
\]
respectively. Combining with (\ref{eq:partial-ad}) and canceling common constant factors yields the desired relation (\ref{eq:desired}).  

It remains to check the statement for $p_2(y)$ directly. Observe that
\[
\Res_k \circ \sum_j \partial_{\mu_i^j} = \partial_{\mu_i} \circ \Res_k,
\]
which implies that
\begin{equation} \label{eq:sub1}
\Res_k\left(\sum_{j_1, j_2} \frac{\partial_{\mu_{i_1}^{j_1}} - \partial_{\mu_{i_2}^{j_2}}}{\mu_{i_1}^{j_1} - \mu_{i_2}^{j_2}} f\right) = k \frac{\partial_{\mu_{i_1}} - \partial_{\mu_{i_2}}}{\mu_{i_1} - \mu_{i_2}} \Res_k(f).
\end{equation}
For a partition $\tau$ with at most $k$ parts, let $m_\tau(\mu_i^j)$ be the monomial symmetric function in $\mu_i^1, \ldots, \mu_i^k$.  Then we see that 
\begin{align}
\Res_k&\left(\Big(\sum_j \partial_{\mu_i^j}^2 - \frac{2}{k} \sum_{j_1 < j_2} \frac{\partial_{\mu_i^{j_1}} - \partial_{\mu_i^{j_2}}}{\mu_i^{j_1} - \mu_i^{j_2}}\Big) m_\tau(\mu_i^j)\right) \nonumber \\
&= \left(\sum_j \tau_j(\tau_j - 1) - \frac{2}{k} \sum_{j_1 < j_2} \frac{1}{2}\Big(\tau_{j_1}(\tau_{j_1} - 1 - \tau_{j_2}) + \tau_{j_2}(\tau_{j_2} - 1 - \tau_{j_1})\Big)\right) k!\,\mu_i^{|\tau| - 2} \nonumber \\
&= \left(\frac{1}{k} \sum_i \tau_i (\tau_i - 1) + \frac{2}{k} \sum_{j_1 < j_2} \tau_{j_1} \tau_{j_2}\right)(\mu_i^j)^{-2} \Res_k(\mu_\lambda(\mu_i^j)) \nonumber \\
&= \frac{1}{k} |\tau|(|\tau| - 1)(\mu_i^j)^{-2} \Res_k(m_\tau(\mu_i^j)) \nonumber \\
&= \frac{1}{k} \partial_{\mu_i}^2 \Res_k(m_\tau(\mu_i^j)). \label{eq:sub2}
\end{align}
Combining (\ref{eq:sub1}) and (\ref{eq:sub2}), the statement for $p_2(y)$ follows by computing
\begin{align*}
\Res_k \circ \overline{L}_{p_2}(-1/k) &= \Res_k \circ \left( \sum_{i, j} \partial_{\mu_i^j}^2 - \frac{2}{k} \sum_{(i_1, j_1) < (i_2, j_2)} \frac{\partial_{\mu_{i_1}^{j_1}} - \partial_{\mu_{i_2}^{j_2}}}{\mu_{i_1}^{j_1} - \mu_{i_2}^{j_2}}\right) \\
&= \Res_k \circ \left( \sum_i \left( \sum_j \partial_{\mu_i^j}^2 - \frac{2}{k} \sum_{j_1 < j_2} \frac{\partial_{\mu_i^{j_1}} - \partial_{\mu_i^{j_2}}}{\mu_i^{j_1} - \mu_i^{j_2}}\right) - \frac{2}{k} \sum_{i_1 \neq i_2} \sum_{j_1, j_2} \frac{\partial_{\mu_{i_1}^{j_1}} - \partial_{\mu_{i_2}^{j_2}}}{\mu_{i_1}^{j_1} - \mu_{i_2}^{j_2}}\right)\\
&= \frac{1}{k} \left(\sum_i \partial_{\mu_i}^2 - 2k \sum_{i_1 \neq i_2} \frac{\partial_{\mu_{i_1}} - \partial_{\mu_{i_2}}}{\mu_{i_1} - \mu_{i_2}}\right) \circ \Res_k \\
&= \frac{1}{k} \overline{L}_{p_2}(-k) \circ \Res_k. \qedhere
\end{align*}
\end{proof}

\begin{remark}
Lemma \ref{lemma:cm-eq} may be extracted from \cite[Proposition 9.5(ii)]{CEE} on representations of the rational Cherednik algebras $H_{1/k, (N - 1)k}$ and $H_{k, N - 1}$.  We give a proof to keep the exposition self-contained.
\end{remark}

\bibliographystyle{alpha}
\bibliography{ho-final}
\end{document}